\documentclass[10pt,reqno]{amsart}
\usepackage{latexsym,amsmath}
\usepackage{times}
\usepackage{amsfonts}
\usepackage{amssymb}
\usepackage{latexsym}
\usepackage{color}
\usepackage[normalem]{ulem}
\usepackage{fullpage}
\usepackage{bbm}

\newenvironment{claimproof}[1]{\noindent{\bf Proof of Claim #1:\@}}{\hfill $\square$\\}

\newcommand\nix{\,\cdot\,}

\newcommand\Forb{\cF}

\newcommand\contig{\triangleleft}
\newcommand\prc{\pi^{\mathrm{rc}}_{k,n,m}}
\newcommand\ppl{\pi^{\mathrm{pl}}_{k,n,m}}

\newcommand{\beq}{\begin{equation}} \newcommand{\eeq}{\end{equation}}

\newcommand\G{\vec H}

\newcommand\hnm{\cH(n,m)}
\newcommand\hknm{H_k(n,m)}

\numberwithin{equation}{section}

\newcommand\br[1]{\left(#1\right)}

\def\vec#1{\mathchoice{\mbox{\boldmath$\displaystyle#1$}}
{\mbox{\boldmath$\textstyle#1$}}
{\mbox{\boldmath$\scriptstyle#1$}}
{\mbox{\boldmath$\scriptscriptstyle#1$}}}

\newcommand{\Zw}{Z_{\omega}}
\newcommand{\Zwni}{Z_{\omega, \nu}^s}
\newcommand{\Zwnie}{Z_{\omega, \nu, \eta}^{s\,(2)}}
\newcommand{\Zwnis}{Z_{\omega, \nu, n^{-3/8}}^{s\,(2)}}

\newcommand{\Aw}{\cA_\w(n)}
\newcommand{\Awni}{\cA_{\w,\nu}^s(n)}
\newcommand{\Bw}{\cB_\w(n)}
\newcommand{\Bwni}{\cB_{\w,\nu}^s(n)}
\newcommand{\Bwnie}{\cB_{\w, \nu, \eta}^s(n)}
\newcommand{\Bwnis}{\cB_{\w, \nu, n^{-3/8}}^s(n)}
\newcommand{\Bwnisr}{\cB_{\w, \nu, n^{-3/8}}^s(n, \rho^0, \rho^1)}

\newcommand{\rwni}{\rho_{\w,\nu}^s}

\newcommand{\dsec}{d_{k,\mathrm{sec}}}
\newcommand{\dcol}{d_{k,\mathrm{col}}}

\DeclareMathOperator{\pr}{\mathbb P}

\newcommand\SIGMA{\vec\sigma}

\newtheorem{definition}{Definition}[section]
\newtheorem{claim}[definition]{Claim}

\newtheorem{remark}[definition]{Remark}
\newtheorem{theorem}[definition]{Theorem}
\newtheorem{lemma}[definition]{Lemma}
\newtheorem{proposition}[definition]{Proposition}
\newtheorem{corollary}[definition]{Corollary}

\newtheorem{fact}[definition]{Fact}

\newcommand\cA{\mathcal{A}}

\newcommand\cB{\mathcal{B}}
\newcommand\cC{\mathcal{C}}

\newcommand\cF{\mathcal{F}}

\newcommand\cE{\mathcal{E}}

\newcommand\cH{\mathcal{H}}
\newcommand\cS{\mathcal{S}}
\newcommand\cT{\mathcal{T}}

\newcommand\cV{\mathcal{V}}

\def\cC{{\mathcal C}}
\def\cE{{\mathcal E}}

\newcommand\eps{\varepsilon}

\newcommand\Var{\mathrm{Var}}
\newcommand\Erw{\mathbb{E}}
\newcommand\w{{\omega}}

\newcommand{\vecone}{\vec{1}}

\newcommand{\Po}{{\rm Po}}

\newcommand{\bink}[2] {{{#1}\choose {#2}}}

\newcommand\ra{\rightarrow}

\newcommand\bc[1]{\left({#1}\right)}
\newcommand\cbc[1]{\left\{{#1}\right\}}

\newcommand\brk[1]{\left\lbrack{#1}\right\rbrack}

\newcommand\norm[1]{\left\|{#1}\right\|}

\newcommand{\whp}{w.h.p.}

\newcommand{\Erdos}{Erd\H{o}s}
\newcommand{\Renyi}{R\'enyi}

\newcommand\Lem{Lemma}
\newcommand\Prop{Proposition}
\newcommand\Thm{Theorem}
\newcommand\Cor{Corollary}
\newcommand\Sec{Section}

\begin{document}

\title{On the number of solutions in random hypergraph 2-colouring}

\author[Felicia Rassmann]{Felicia Rassmann$^*$}
\thanks{$^*$The research leading to these results has received funding from the European Research Council under the European Union's Seventh Framework
			Programme (FP/2007-2013) / ERC Grant Agreement n.\ 278857--PTCC}
\date{\today} 

\address{Felicia Rassmann, {\tt rassmann@math.uni-frankfurt.de}, Goethe University, Institute of Mathematics, 10 Robert-Mayer-Str, Frankfurt 60325, Germany.}

\begin{abstract}
\noindent We determine the limiting distribution of the logarithm of the number of satisfying assignments in the random $k$-uniform hypergraph 2-colouring problem in a certain density regime for all $k\ge 3$.
As a direct consequence we obtain that in this regime the random colouring model is contiguous wrt.~the planted model, a result that helps simplifying the transfer of statements between these two models.  

\end{abstract}

\maketitle
\section{Introduction}

\subsection{Background and motivation}
A main focus when studying random constraint satisfaction problems is on determining the expected value of the number of solutions and understanding how this number evolves when the constraint density changes.
However, up to now the distribution of the number of solutions remains elusive in any of the standard examples of random constraint satisfaction problems. 

In this paper we consider {\em random $k$-uniform hypergraphs} $H_k(n,m)$ on the vertex set $\brk{n}=\{1,\ldots,n\}$ with exactly $m$ hyperedges each comprising of $k$ distinct vertices and chosen uniformly at random from all possible subsets of $\brk{n}$ of size $k$. The {\em hypergraph 2-colouring problem} is a random constraint satisfaction problem  where one is interested in the number $Z(\hknm)$ of {\em 2-colourings} (also called {\em solutions}) of $\hknm$, which are maps $\sigma: \brk{n} \to \{0,1\}$ that generate no \textit{monochromatic} edges (i.e. edges $e$ such that $|\sigma(e)|=1$).

In the following we only consider \textit{sparse} random hypergraphs, meaning that $m=O(n)$ as $n\to\infty$. We call the parameter $d=km/n$ the \textit{hyperedge density}.
 
As for many other random constraint satisfaction problems, there is a conjecture as to a sharp threshold for the existence of solutions in terms of the hyperedge density. The best currently known bounds on this threshold $\dcol$ are from Achlioptas and Moore \cite{Achlioptas} and Coja-Oghlan and Zdeborov\'a \cite{Lenka}. Furthermore, there was a prediction for this density by statistical physicists \cite{dallasta, kmrsz} suggesting that
$$\dcol/k = 2^{k-1}\ln 2-\frac{\ln 2}2-\frac 14+\eps_k \qquad \text{with } \lim_{k\to\infty}\eps_k \to 0.$$

This prediction was proved by Coja-Oghlan and Panagiotou \cite{Panagiotou} for the problem of NAE-$k$-SAT which is almost equivalent to hypergraph 2-colouring and it should be possible to transfer the result without major difficulties.

For a long period of time, the best rigorous upper and lower bounds on $\dcol$ were based on the non-constructive {\em first} and {\em second moment method} applied to the random variable $Z$. Achlioptas and Moore \cite{Achlioptas} proved that there is a critical density $\dsec$ such that $\Erw\brk{Z^2}\le C\cdot \Erw\brk{Z}^2$ for some constant $C=C(d,k)>0$ for all $d<\dsec$ but is violated for $d>\dsec$. Via the Paley-Zygmund inequality it can be established that $\dsec$ is a non-constructive lower bound on $\dcol$. In~\cite{Achlioptas} it was shown that
\begin{align*}
 \dsec/k&=2^{k-1}\ln 2-\frac{\ln 2}2-\frac 12+\eps_k,\qquad\mbox{ with }\lim_{k\ra\infty}\eps_k=0.
\end{align*} 


\subsection{Results}

The main result in the present paper is to show that under certain conditions the number $Z(\hknm)$ of $2$-colourings of the random $k$-uniform hypergraph
is concentrated remarkably tightly and to even obtain the distribution of $\ln Z(\hknm)-\ln \Erw{\brk{Z(\hknm)}}$ asymptotically in a density regime essentially up to $\dsec$. 

As our computations have to be very precise in order to obtain these results, we need to distinguish between the quantity $d^{\prime}$, which is a fixed number such that $m = \lceil d^{\prime}n/k \rceil$, and the quantity $d = km/n$, which arises naturally in the computations of the first and second moment. We note that the quantity $d = d(n)$ depends on $n$, whereas $d^{\prime}$ is assumed to be fixed as $n\ra\infty$. However, it is elementary to show that $d^{\prime} \sim d$.
We also always require that $k\ge 3$.

\begin{theorem}\label{Thm_main}
Let $k\geq3$ and $d^{\prime}/k\leq 2^{k-1}\ln 2-2$ as well as
\begin{align*}
	\lambda_l=\frac{[d(k-1)]^l}{2l}\quad\mbox{ and }\quad\delta_l=\frac{(-1)^l}{\br{2^{k-1}-1}^l}
\end{align*}
for $l\ge 2$. Further let $(X_l)_l$ be a family of independent Poisson variables with $\Erw[X_l]=\lambda_l$, all defined on the same probability space. Then the random variable
	$$W=\sum_{l}X_l\ln(1+\delta_l)-\lambda_l\delta_l$$
satisfies $\Erw |W|<\infty$ and $\ln Z(\hknm)-\ln \Erw{\brk{Z(\hknm)}}$ converges in distribution to $W$.
\end{theorem}

\begin{remark}
By definition, $W$ has an infinitely divisible distribution.
It was shown in \cite{Janson} that the random variable $W^{\prime}=\exp\brk{W}$ converges almost surely and in $L^2$ with $\Erw\brk{W^{\prime}}=1$ and $\Erw\brk{{W^{\prime}}^2}=\exp\brk{\sum_l\lambda_l\delta_l^2}$. Thus, by Jensen's inequality it follows that $\Erw\brk{W}\le 0 $ and $\Erw\brk{W^2}\le \sum_l\lambda_l\delta_l^2$. 
\end{remark}

As a direct consequence of \Thm~\ref{Thm_main} we obtain the following. 
\begin{corollary}\label{Cor_conc}
Assume that $k\geq3$ and $d^{\prime}/k\leq 2^{k-1}\ln 2-2$. Then
	\begin{equation}\label{eq_Cor_1}
	 \lim_{\omega\ra\infty}\lim_{n\ra\infty}\pr\brk{|\ln Z(\hknm)-\ln\Erw\brk{Z(\hknm}|\leq\omega}=1.
	\end{equation}

On the other hand, for any fixed number $\omega>0$ we have
	$$\lim_{n\ra\infty}\pr\brk{|\ln Z(\hknm)-\ln\Erw[Z(\hknm)]|\leq\omega}<1.$$
\end{corollary}

For $d^{\prime},k$ covered by \Cor~\ref{Cor_conc} we have $\ln Z_k(G(n,m))=\Theta(n)$ \whp. Thus, it would be reasonable to expect that  $\ln Z_k(G(n,m))$ has fluctuations of order e.~g.~ $\sqrt n$. However, the first part of \Cor~\ref{Cor_conc} shows that actually $\ln Z_k(G(n,m))$ fluctuates \whp~by no more than $\omega=\omega(n)$ for {\em any} $\omega(n)\ra\infty$.
Moreover, the second part shows that this is best possible.

\subsubsection{Planted model and silent planting}

When proving results about random hypergraphs and investigating their properties, it turns out very useful and often essential to have the notion of \textit{typical} 2-colourings at hand. By a typical 2-colouring of a random hypergraph we mean a 2-colouring chosen uniformly at random from all its 2-colourings. To make this formal, let $\Lambda_{k,n,m}$ be the set of all pairs $(H_k,\sigma)$ with $H_k=H_k(n,m)$ and $\sigma$ a $2$-colouring of $H_k$. Now we define a probability distribution $\prc[H_k,\sigma]$ on $\Lambda_{k,n,m}$ by letting
$$\prc[H_k,\sigma]=\brk{Z(H_k)\bink{\bink{n}k}m\pr\brk{H_k\mbox{ is $2$-colourable}}}^{-1}.$$

We call this distribution the {\em random colouring model} or {\em Gibbs distribution}. It can also be described as the distribution produced by the following experiment.
\begin{description}
	\item[RC1] Generate a random hypergraph $H_k=\hknm$ provided that $Z(H_k)>0$.
	\item[RC2] Choose a $2$-colouring $\sigma$ of $H_k$ uniformly at random.  
	The result of the experiment is $(H_k,\sigma)$.
\end{description}

However, up to now there is no known method to implement this experiment efficiently for a wide range of hyperedge densities. In fact, the first step {\bf RC1} is easy to process, because we are only interested in values of $d$ where $\hknm$ is 2-colourable \whp~and thus the conditioning on $Z(H_k)>0$ does not cause problems (the probability $\pr\brk{\hknm\mbox{ is $2$-colourable}}$ is close to 1). 
But what turns the direct study of the distribution $\prc$ into a challenge is step {\bf RC2}, because in the interesting density regimes we cannot even find one 2-colouring algorithmically let alone sample one uniformly: The currently best-performing algorithms for sampling a $2$-colouring of $\hknm$ are known to succeed up to density $d^{\prime}/k=c 2^{k-1}/k$ for some constant $c>0$ \cite{algofail}, which is about a factor of $k$ below the colourability threshold.

To circumvent these difficulties, we consider an alternative probability distribution
on $\Lambda_{k,n,m}$ called the {\em planted model}, which is much easier to approach. To describe this experiment, for $\sigma:[n]\ra\{0,1\}$ let
$$\Forb(\sigma)=\bink{|\sigma^{-1}(0)|}k+\bink{|\sigma^{-1}(1)|}k$$
be the number of hyperedges of the complete hypergraph that are monochromatic under $\sigma$. Then the planted distribution is induced by the following experiment:
\begin{description}
	\item[PL1] Choose a map $\SIGMA:\brk n\ra\{0,1\}$ uniformly at random provided that that $\Forb(\SIGMA)\leq\bink nk- m$.
	\item[PL2] Generate a $k$-uniform hypergraph $\G$ on $\brk n$ consisting of $m$ hyperedges that are bichromatic under $\SIGMA$ uniformly at random.
	The result of the experiment is $(\G,\SIGMA)$.
\end{description}
Thus, the probability that the planted model assigns to a pair $(H_k,\sigma)$ is
$$\ppl[H_k,\sigma]\sim\brk{2^n\bink{\bink{n}k}m\pr\brk{\mbox{$\sigma$ is a $2$-colouring of $H_k$}}}^{-1}.$$
We observe that in contrast to the ''difficult'' step {\bf RC2}, step {\bf PL2} is much easier to implement.

Of course, the two probability distributions $\prc$ and $\ppl$ differ. Under $\prc$, the hypergraph is chosen uniformly at random, whereas under $\ppl$ it comes up with a probability that is proportional to its number of solutions, meaning that hypergraphs exhibiting many 2-colourings are ``favoured`` by the planted model.

However, the two models are related if $m=m(n)$ is such that
\begin{equation}\label{eq_FriedgutConcentration}
\ln Z(\hknm)=\ln\Erw[Z(\hknm)]+o(n)\qquad\mbox\whp
\end{equation}

Coja-Oghlan an Achlioptas showed in \cite{Barriers} (where it was actually formulated for the problem of graph $k$-colouring, but the authors asserted that it also holds for hypergraph 2-colouring), that if~(\ref{eq_FriedgutConcentration}) is satisfied, then the following is true.
\begin{equation}\label{eq_plantingTrick} 
\parbox{15cm}{If $(\cE_n)$ is a sequence of events $\cE_n\subset\Lambda_{k,n,m}$
	such that $\ppl[\cE_n]\leq\exp\brk{-\Omega(n)}$, then $\prc[\cE_n]=o(1)$.}
\end{equation}

The statement~(\ref{eq_plantingTrick}) was baptised ``quiet planting'' by Krzakala and Zdeborov\'a~\cite{QuietPlanting} and has ever since been used to study the behaviour of the set of colourings and its geometrical structure in various random constraint satisfaction problems \cite{Barriers,Cond,Molloy,MolloyRestrepo,Reconstr}.
Yet a significant complication in the use of~(\ref{eq_plantingTrick}) is that 
$\cE_n$ must not only have a small probability but is required to be {\em exponentially} unlikely in the planted model.
This has caused substantial difficulties in several applications (e.g., \cite{Bapst,Cond,Molloy}).

\Thm~\ref{Thm_main} enables us to establish a very strong connection between the random colouring model and the planted model.
To state this, we recall the following definition.
Suppose that $\vec\mu=(\mu_n)_{n\geq1},\vec\nu=(\nu_n)_{n\geq1}$ are two sequences of probability measures such that
$\mu_n,\nu_n$ are defined on the same probability space $\Omega_n$ for every $n$.
Then $(\mu_n)_{n\geq1}$ is {\em contiguous} with respect to $(\nu_n)_{n\geq1}$, in symbols $\vec\mu\contig\vec\nu$, if
for any sequence $(\cE_n)_{n\geq1}$ of events such that $\lim_{n\ra\infty}\nu_n(\cE_n)=0$ we have $\lim_{n\ra\infty}\mu_n(\cE_n)=0$.

We show that as a consequence of (\ref{Thm_main}) the statement~(\ref{eq_plantingTrick}) can be sharpened in the strongest possible sense.
Roughly speaking, we are going to show that in a density regime nearly up to the second moment lower bound the random colouring model is contiguous with respect to the planted model,
	i.e., in~(\ref{eq_plantingTrick}) it suffices that $\ppl[\cE_n]=o(1)$.
	
\begin{corollary}\label{Cor_cont}
Assume that $d/k\leq 2^{k-1}\ln 2-2$. Then
	$(\prc)_{n\geq1}\contig(\ppl)_{n\geq1}.$
\end{corollary}

As done in \cite{aco_plantsil}, we refer to this contiguity statement as \textit{silent planting}.

\subsection{Discussion and further related work.}\label{Sec_discussion}

The ideas for the proofs follow the way beaten in \cite{aco_plantsil}, where statements analogue to~\Cor~\ref{Cor_conc} and~\Cor~\ref{Cor_cont} are shown for the problem of $k$-colouring random graphs. However, \Thm~\ref{Thm_main} is stronger than the results obtained in \cite{aco_plantsil} because we determine the exact distribution of $\ln Z-\Erw\brk{\ln Z}$ asymptotically. 
The main observation used in the proofs is that the variance in the logarithm of the number of 2-colourings can be attributed to the fluctuations of the number of cycles of bounded length. The same phenomenon was observed in \cite{aco_plantsil} and also in \cite{aco_Wormald}, where a combination of the second moment method and small subgraph conditioning was applied to derive a result similar to ours for the problem of random regular $k$-SAT.
 
Small subgraph conditioning was originally developed by Robinson and Wormald in \cite{RobinsonWormald} to investigate the Hamiltonicity of random regular graphs of degree at least three. Janson showed in \cite{Janson} that the method can be used to obtain limiting distributions.  Neeman und Netrapalli \cite{Neeman} used the method to obtain a result on non-distinguishability of the \Erdos-\Renyi~ model and the stochastic block model. Moore \cite{Moore} used the method to determine the satisfiability threshold for positive 1-in-$k$-SAT, a Boolean satisfiability problem where each clause contains $k$ variables and demands that exactly one of them is true.

Similar to \cite{Janson}, we aim at obtaining a limiting distribution. Unfortunately, Jansons result does not apply directly in our case for the following reason. Since verifying the required properties to apply small subgraph conditioning directly for the random variable $Z$ is very intricate, we break $Z$ down into $Z=\sum_s Z^s$ for some number $s>0$ and determine the first and second moment of these smaller variables $Z^s$. However, it is not evident how to apply Jansons result to that growing number of variables simultaneously. Instead, we choose to perform a variance analysis along the lines of \cite{RobinsonWormald}. The same approach was pursued in \cite{aco_Wormald}, and thus our proof technique is similar to theirs in flavour, but we have the advantage of only having to deal with a very moderately growing number $s$ of variables, which simplifies matters slightly.

\subsection{Preliminaries and notation}
We always assume that $n\geq n_0$ is large enough for our various estimates to hold and denote by $[n]$ the set $\{1,...,n\}$. 

We use the standard $O$-notation when referring to the limit $n\ra\infty$.
Thus, $f(n)=O(g(n))$ means that there exist $C>0$, $n_0>0$ such that for all $n>n_0$ we have $|f(n)|\leq C\cdot|g(n)|$.
In addition, we use the standard symbols $o(\cdot),\Omega(\cdot),\Theta(\cdot)$.
In particular, $o(1)$ stands for a term that tends to $0$ as $n\ra\infty$.
Furthermore, the notation $f(n)\sim g(n)$ means that $f(n)=g(n)(1+o(1))$ or equivalently $\lim_{n\to\infty} f(n)/g(n)=1$. Besides taking the limit $n\to\infty$, at some point we need to consider the limit $\nu\to\infty$ for some number $\nu \in \mathbb N$. Thus, we introduce $f(n,\nu)\sim_{\nu} g(n,\nu)$ meaning that $\lim_{\nu\to\infty}\lim_{n\to\infty} f(n,\nu)/g(n,\nu)=1$.

If $p=(p_1,\ldots,p_l)$ is a vector with entries $p_i\geq0$, then we let
	$$\cH(p)=-\sum_{i=1}^lp_i\ln p_i.$$
Here and throughout, we use the convention that $0\ln0=0$.
Hence, if $\sum_{i=1}^lp_i=1$, then $\cH(p)$ is the entropy of the probability distribution $p$.
Further, for a number $x$ and an integer $h>0$ we let $(x)_h=x(x-1)\cdots(x-h+1)$ denote the $h$th falling factorial of $x$.

For the sake of simplicity we choose to prove \Thm~\ref{Thm_main} using the random hypergraph model $\hnm$.
This is a random $k$-uniform (multi-)hypergraph on the vertex set $[n]$ obtained by choosing $m$ hyperedges $\vec e_1,\ldots,\vec e_m$ of the complete hypergraph on $n$ vertices
uniformly and independently at random (i.e., with replacement). In this model we may choose the same edge more than once, however, the following statement shows that this is quite unlikely.
 
\begin{fact}\label{Fact_doubleEdge}
Assume that $m=m(n)$ is a sequence such that $m=O(n)$ and let $\cA$ be the event that $\hnm$ has no multiple edges. Then $\pr\brk{\neg\cA}=O(n^{2-k})$.
\end{fact}
\noindent All statements and all proofs are valid for any $k\ge 3$. 

\section{Outline of the proof}\label{Sec_outline}

We classify the $2$-colourings according to their proportion of assigned colours: For a map $\sigma:[n]\to \{0,1\}$ we define 
\begin{align}\label{rhosigma}
\rho(\sigma)= | \sigma^{-1}(0) |/n
\end{align}
and call this value the {\em colour density} of $\sigma$. We let $\cA(n)$ signify the set of all possible colour densities $\rho(\sigma)$ for $\sigma:\brk n\ra \{0,1\}$. We will later show that when bounding the moments of $Z(\hnm)$ we can confine ourselves to colourings such that the proportion of the two colours does not deviate too much from 1/2. 
Formally, we say that $\rho\in[0,1]$ is {\em $(\w,n)$-balanced} for $\w\in\mathbb N$ if
\begin{align*}
\rho\in \left[\frac 12- \frac{\w}{\sqrt{n}},\frac 12+ \frac{\w}{\sqrt{n}}\right)
\end{align*}
and we denote by $\Aw$ the set of all $(\w,n)$-balanced colour densities $\rho\in\cA(n)$. For a hypergraph $H$ on $[n]$ we let $\Zw(H)$ signify the number of {\em $(\w,n)$-balanced} colourings, which are 2-colourings $\sigma$ such that $\rho(\sigma)\in \Aw$. 
As we will see, it will turn out useful to split up the set $\Aw$ into smaller sets in the following way. For $\nu \in \mathbb N$ and $s\in \brk{\w\nu}$ let
\begin{align}\label{rwni}
\rwni=\frac 12-\frac{\w}{\sqrt n}+\frac {2s-1}{\nu\sqrt n}.
\end{align}
Let $\Awni$ be the set of all colour densities $\rho\in\cA(n)$ such that
\begin{align*}
\rho\in \left[\rwni- \frac{1}{\nu\sqrt{n}},\rwni+ \frac{1}{\nu\sqrt{n}}\right).
\end{align*}
For a hypergraph $H$ let $\Zwni(H)$ denote the number of 2-colourings $\sigma$ of $H$ such that $\rho(\sigma)\in\Awni$. We are going to apply small subgraph conditioning to $\Zwni$ rather than directly to $Z$. We observe that for each fixed $\nu$ we have $\Zw=\sum_{s=1}^{\w\nu} \Zwni$. In \Sec~\ref{Sec_first_moment} we will calculate the first moments of $Z$ and $\Zw$ to obtain the following.

\begin{proposition} \label{Prop_first_moment}
	Let $k\geq 3, d^{\prime}\in (0, \infty)$ and $\omega>0$.
	Then
	$$
	\Erw\brk{Z(\hnm)}=\Theta\br{2^n\br{1-2^{1-k}}^m}\quad\mbox{and}\quad
	\lim_{\w\to\infty}\liminf_{n\to\infty}\frac{ \Erw\brk{\Zw(\hnm)}}{\Erw\brk{Z(\hnm)}}=1.$$
\end{proposition}

As outlined in \Sec~\ref{Sec_discussion}, our basic strategy is to show that the fluctuations
of $\ln Z$ can be attributed to fluctuations in the number of cycles of a bounded length.
Hence, for an integer $l\geq 2$ we let $C_{l,n}$ denote the number of cycles of length (exactly) $l$ in $\hnm$.
Let
	\begin{equation}\label{eq_lambdadelta}
	\lambda_l=\frac{[d(k-1)]^l}{2l}\quad\mbox{ and }\quad\delta_l=\frac{(-1)^l}{\br{2^{k-1}-1}^l}.
	\end{equation}
We will see that $\lambda_l$ denotes the expected number of cycles of length $l$ in a random $k$-uniform hypergraph, whereas $\delta_l$ is a correction factor taking into account that we only allow for bichromatic edges. 
It is well-known that $C_{2,n},\ldots$ are asymptotically independent Poisson variables \cite[\Thm~5.16]{Bollobas}.
More precisely, we have the following.

\begin{fact}\label{Fact_cycles}
If $c_2,\ldots,c_L$ are non-negative integers, then
	$$\lim_{n\ra\infty}\pr\brk{\forall 2\leq l\leq L:C_{l,n}=c_l}=\prod_{l=2}^L\pr\brk{\Po(\lambda_l)=c_l}.$$
\end{fact}

Next, we investigate the impact of the cycle counts $C_{l,n}$ on the first moment of $\Zwni$. In \Sec~\ref{Sec_short_cycles} we prove the following.

\begin{proposition}\label{Prop_ratio_first_cond}
Assume that $k\geq3$ and $d^{\prime}\in(0,\infty)$. 
Then
	\begin{equation}\label{eq_lambda_delta_conv}
	\sum_{l=2}^\infty\lambda_l\delta_l^2<\infty.
	\end{equation}
Moreover, let $\w,\nu \in \mathbb N$. If $c_2,\ldots, c_L$ are non-negative integers, then for any $s\in\brk{\w\nu}$:
\begin{align}\label{eq_prop_ratio_first_cond}
\frac{\Erw\brk{\Zwni(\hnm)|\forall 2\leq l\leq L:C_{l,n}=c_l}}{\Erw\brk{\Zwni(\hnm)}}\sim \prod_{l=2}^L\brk{1+\delta_l}^{c_l}\exp\brk{-\delta_l\lambda_l}.
\end{align}
\end{proposition}

Additionally, we need to know the second moment of $\Zwni$ very precisely. The following proposition is the key result of our approach and the one that requires the most technical work. Its proof can be found at the end of \Sec~\ref{Sec_second_moment}.


\begin{proposition}\label{Prop_ratio_second_first}
	Assume that $k \geq 3$ and $d^{\prime}/k < 2^{k-1}\ln 2-2$ and let $\w,\nu \in \mathbb N$. Then for every $s\in\brk{\w\nu}$ we have
	$$\frac{ \Erw \brk{\Zwni(\hnm)^2}}{\Erw \brk{\Zwni(\hnm)}^2} \sim_{\nu} \exp \brk{\sum_{l \geq 2} \lambda_l \delta_l^2}.  $$
\end{proposition}

\noindent We now derive \Thm~\ref{Thm_main} from \Prop s~\ref{Prop_first_moment}-\ref{Prop_ratio_second_first}. The key observation we will need is that the variance of the random variables $\Zwni$ can almost entirely be attributed to the fluctuations of the number of short cycles. As done in \cite{aco_Wormald}, the arguments we use are similar to the small subgraph conditioning from \cite{Janson, RobinsonWormald}. But we do not refer to any technical statements from \cite{Janson, RobinsonWormald} directly because instead of working only with the random variable $Z$ we need to control all $\Zwni$ for fixed $\w,\nu \in \mathbb N$ simultaneously. In fact, ultimately we have to take $\nu\to\infty$ and $\w\to\infty$ as well.
Our line of argument follows the path beaten in \cite{aco_Wormald} and the following three lemmas are an adaption of the ones there.

For $L>2$ let $\cF_{L}=\cF_{L,n}(d,k)$ be the $\sigma$-algebra generated by the random variables $C_{l,n}$ with $2\le l\le L$.
For each $L \ge 2$ the standard decomposition of the variance yields
$$\Var\brk{\Zwni(\hnm)}=\Var\brk{\Erw\brk{\Zwni(\hnm)|\cF_L}}+\Erw\brk{\Var\brk{\Zwni(\hnm)|\cF_L}}.$$
The term $\Var\brk{\Erw\brk{\Zwni(\hnm)|\cF_L}}$ accounts for the amount of variance induced by the fluctuations of the number of cycles of length at most $L$. The strategy when using small subgraph conditioning is to bound the second summand, which is the expected conditional variance
$$\Erw\brk{\Var\brk{\Zwni(\hnm)|\cF_L}}=\Erw\brk{\Erw\brk{{\Zwni(\hnm)}^2|\cF_L}-\Erw\brk{\Zwni(\hnm)|\cF_L}^2}.$$
In the following lemma we show that in fact in the limit of large $L$ and $n$ this quantity is negligible. This implies that conditioned on the number of short cycles the variance vanishes and thus the limiting distribution of $\ln \Zwni$ is just the limit of $\ln\Erw\brk{\Zwni|\cF_L}$ as $n,L \to \infty$. This limit is determined by the joint distribution of the number of short cycles.
\begin{lemma}\label{Lem_conc_1}
 For $d^{\prime} \in (0,\infty)$ and any $\w,\nu \in \mathbb N$ and $s\in[2\omega\nu]$ we have
\begin{align*}
 \limsup_{L\to\infty} \limsup_{n\to\infty}\Erw\brk{\frac{\Erw\brk{{\Zwni(\hnm)}^2|\cF_L}-\Erw\brk{\Zwni(\hnm)|\cF_L}^2}{\Erw\brk{\Zwni(\hnm)}^2}}=0.
\end{align*}
\end{lemma}
\begin{proof}
Fix $\w, \nu\in \mathbb N$ and set $Z_s=\Zwni(\hnm)$. Using Fact~\ref{Fact_cycles} and equation (\ref{eq_prop_ratio_first_cond}) from \Prop~\ref{Prop_ratio_first_cond} we can choose for any $\eps>0$ a constant $B=B(\eps)$ and $L\ge L_0(\eps)$ large enough such that for each large enough $n\ge n_0(\eps, B, L)$ we have for any $s\in\brk{\w\nu}$:
 \begin{align}\label{eq_lem_conc_1_1}
\nonumber  \Erw\brk{\Erw\brk{Z_s|\cF_L}^2}&\ge \sum_{c_1,...,c_L\le B} \Erw\brk{Z_s|\forall 2 \le l\le L: C_{l,n}=c_l}^2 \pr\brk{\forall 2 \le l\le L: C_{l,n}=c_l}\\
\nonumber   &\ge \exp\brk{-\eps}\Erw\brk{Z_s}^2 \sum_{c_1,...,c_L\le B} \ \prod_{l=2}^L \brk{(1+\delta_l)^{c_l}\exp\brk{-\lambda_l\delta_l}}^2\pr\brk{\Po(\lambda_l)=c_l}\\
\nonumber   &= \exp\brk{-\eps}\Erw\brk{Z_s}^2 \sum_{c_1,...,c_L\le B} \ \prod_{l=2}^L \frac{\brk{(1+\delta_l)^2\lambda_l}^{c_l}}{c_l!\exp\brk{2\lambda_l\delta_l+\lambda_l}}\\
            &\ge \Erw\brk{Z_s}^2 \exp\brk{-2\eps+ \sum_{l=2}^L \delta_l^2\lambda_l}.
 \end{align}
The tower property for conditional expectations and the standard formula for the decomposition of the variance yields
\begin{align*}
 \Erw\brk{Z_s^2}&=\Erw\brk{\Erw\brk{Z_s^2|\cF_L}}=\Erw\brk{\Erw\brk{Z_s^2|\cF_L}-\Erw\brk{Z_s|\cF_L}^2}+ \Erw\brk{\Erw\brk{Z_s|\cF_L}^2}
\end{align*}
and thus, using (\ref{eq_lem_conc_1_1}) we have
\begin{align}\label{eq_lem_conc_1_2}
\frac{\Erw\brk{\Erw\brk{Z_s^2|\cF_L}-\Erw\brk{Z_s|\cF_L}^2}}{\Erw\brk{Z_s}^2 }\le \frac{ \Erw\brk{Z_s^2}}{\Erw\brk{Z_s}^2}- \exp\brk{-2\eps+ \sum_{l=2}^L \delta_l^2\lambda_l}.
\end{align}
Finally, the estimate $\exp[-x]\ge 1-x$ for $|x|<1/8$ combined with (\ref{eq_lem_conc_1_2}) and \Prop~\ref{Prop_ratio_second_first} implies that for large enough $\nu,n,L$ and each $s\in\brk{\w\nu}$ we have
\begin{align*}
 \frac{\Erw\brk{\Erw\brk{Z_s^2|\cF_L}-\Erw\brk{Z_s|\cF_L}^2}}{\Erw\brk{Z_s}^2 }\le 2\eps \exp\brk{\sum_{l=2}^\infty \delta_l^2\lambda_l}.
\end{align*}
As this holds for any $\eps>0$ and by equation (\ref{eq_lambda_delta_conv}) from \Prop~\ref{Prop_ratio_first_cond} the expression $\exp\brk{\sum_{l=2}^\infty \delta_l^2\lambda_l}$ is bounded, the proof of the lemma is completed by first taking $n\to\infty$ and then $L\to\infty$.
\end{proof}

\begin{lemma}\label{Lem_conc_3}
	For $d \in (0,\infty)$ and any $\alpha>0$ we have
	\begin{align*}
	\limsup_{L\to\infty}\limsup_{n\to\infty} \pr\brk{|Z(\hnm)-\Erw\brk{Z(\hnm)|\cF_L}|>\alpha\Erw\brk{Z(\hnm)}}=0.
	\end{align*}
\end{lemma}

\begin{proof}
	To unclutter the notation, we set $Z=Z(\hnm)$ and $\Zw=\Zw(\hnm)$. First we observe that \Prop~\ref{Prop_first_moment} implies that for any $\alpha>0$ we can choose $\w \in \mathbb N$ large enough such that
	\begin{align}\label{eq_lem_conc_3_1}
	\liminf_{n\to\infty} \Erw\brk{\Zw}>(1-\alpha^2)\Erw\brk Z.
	\end{align}
	We let $\nu \in \mathbb N$. To prove the statement, we need to get a handle on the cases where the random variables $\Zwni(\hnm)$ deviate strongly from their conditional expectation $\Erw\brk{\Zwni(\hnm)|\cF_L}$. We let $Z_s=\Zwni(\hnm)$ and define 
	\begin{align*}
	X_s=|Z_s-\Erw\brk{Z_s|\cF_L}|\cdot \mathbf{1}_{\{|Z_s-\Erw\brk{Z_s|\cF_L}|>\alpha \Erw\brk{Z_s}\}}
	\end{align*}
	and $X=\sum_{s=1}^{\w\nu} X_s$. Then these definitions directly yield 
	\begin{align}\label{eq_lem_conc_3_2}
	\pr\brk{X<\alpha  \Erw\brk{\Zw}}\le \pr\brk{\left|\Zw-\Erw\brk{\Zw|\cF_L}\right|<2\alpha \Erw\brk{\Zw}}.
	\end{align}
	By the definition of the $X_s$'s and Chebyshev's inequality it is true for every $s$ that
	\begin{align*}
	\Erw\brk{X_s|\cF_L}&\le \sum_{j\ge 0} 2^{j+1} \alpha \Erw\brk{Z_s}\pr\brk{\left|Z_s-\Erw\brk{Z_s|\cF_L}\right|>2^j\alpha \Erw\brk{Z_s}}\le \frac{4\Var\brk{Z_s|\cF_L}}{\alpha \Erw\brk{Z_s}}.
	\end{align*}
	Hence, using that with \Prop~\ref{Prop_first_moment} there is a number $\beta=\beta(\alpha, \w)$ such that $\Erw\brk{Z_s}/\Erw\brk{Z}\le \beta/(\w\nu)$ for all $s\in [\w\nu]$ and $n$ large enough, we have
	\begin{align*}
	\Erw\brk{X|\cF_L}\le \sum_{s=1}^{\w\nu}\frac{4\Var\brk{Z_s|\cF_L}}{\alpha \Erw\brk{Z_s}}\le  \frac{2\beta \Erw\brk Z}{\alpha\nu\w} \sum_{s=1}^{\w\nu}\frac{\Var\brk{Z_s|\cF_L}}{\Erw\brk{Z_s}^2}.
	\end{align*}
	Taking expectations, choosing $\eps=\eps(\alpha, \beta, \w)$ small enough and applying \Lem~\ref{Lem_conc_1}, we obtain
	\begin{align}\label{eq_lem_conc_3_4}
	\Erw\brk X =\Erw\brk{\Erw\brk{X|\cF_L}} \le \frac{2\beta \Erw\brk Z}{\alpha\nu\w} \sum_{s=1}^{\w\nu}\frac{\Erw\brk{\Var\brk{Z_s|\cF_L}}}{\Erw\brk{Z_s}^2} \le \frac{4\beta\eps\Erw\brk Z}{\alpha} \le \alpha^2\Erw\brk Z.
	\end{align}
	Using (\ref{eq_lem_conc_3_2}), Markov's inequality, (\ref{eq_lem_conc_3_4}) and (\ref{eq_lem_conc_3_1}), it follows that
	\begin{align}\label{eq_lem_conc_3_5}
	\pr\brk{\left|\Zw-\Erw\brk{\Zw|\cF_L}\right|<2\alpha \Erw\brk{\Zw}}\ge 1-2\alpha.
	\end{align}
	
	\noindent Finally, the triangle inequality combined with Markov's inequality and equations (\ref{eq_lem_conc_3_1}) and (\ref{eq_lem_conc_3_5}) yields
	\begin{align*}
	\pr\brk{\left|Z-\Erw\brk{Z|\cF_L}\right|>\alpha\Erw\brk{Z}} & \le\pr\brk{\left|Z-\Zw\right|+\left|\Zw-\Erw\brk{\Zw|\cF_L}\right|+\left|\Erw\brk{\Zw|\cF_L}-\Erw\brk{Z|\cF_L}\right|>\alpha\Erw\brk{Z}}\\
	&\le 3\alpha + \alpha/3 + 3\alpha < 7\alpha,
	\end{align*}
	which proves the statement. 
\end{proof}

\begin{lemma}\label{Lem_conc_2}
	Let 
	\begin{align}\label{eq_lem_conc_2_0}
	U_L&=\sum_{l=2}^L C_{l,n}\ln(1+\delta_l)-\lambda_l\delta_l.
	\end{align}
	Then $\limsup_{L\to\infty} \limsup_{n\to\infty} \Erw\brk{|U_L|}<\infty$ and further for any $\eps>0$ we have
	
	\begin{align}\label{eq_lem_conc_2_1}
	\limsup_{L\to\infty}\limsup_{n\to\infty} \pr\brk{|\ln \Erw\brk{Z(\hnm)|\cF_L}-\ln \Erw\brk{Z(\hnm)}-U_L|>\eps}=0
	\end{align}
\end{lemma}
\begin{proof}
	In a first step we show that $\Erw\brk{|U_L|}$ is uniformly bounded. As $x-x^2\le\ln(1+x)\le x$ for $|x|\le 1/8$ we have for every $l\le L$: 
	\begin{align*}
	\Erw\brk{\left|C_{l,n}\ln(1+\delta_l)-\lambda_l\delta_l\right|}\le \delta_l\Erw\brk{\left|C_{l,n}-\lambda_l\right|}+\delta_l^2\Erw\brk{C_{l,n}}.
	\end{align*}
	Therefore, Fact~\ref{Fact_cycles} implies that
	\begin{align}\label{eq_lem_conc_2_2}
	\Erw\brk{\left|U_L\right|}\le\sum_{l=2}^L\delta_l\sqrt{\lambda_l}+\delta_l^2\lambda_l.
	\end{align}
	\Prop~\ref{Prop_ratio_first_cond} ensures that $\sum_l \delta_l^2\lambda_l <\infty$. Furthermore, as we are in the regime $d^{\prime}/k\le 2^{k-1} \ln 2$, we have $\sum_l \delta_l\sqrt{\lambda_l}\le \sum_l k^l 2^{-(k-1)l/2} <\infty$ and thus (\ref{eq_lem_conc_2_2}) shows that $\Erw\brk{|U_L|}$ is uniformly bounded. 
	
	To prove (\ref{eq_lem_conc_2_1}), for given $n$ and a constant $B>0$ we let $\cC_B$ be the event that $C_{l,n}<B$ for all $l\le L$. Referring to Fact~\ref{Fact_cycles}, we can find for each $L,\eps>0$ a $B>0$ such that
	\begin{align}\label{eq_lem_conc_2_3}
	\pr\brk{\cC_B}>1-\eps.
	\end{align}
	To simplify the notation we set $Z=Z(\hnm)$ and $\Zw=\Zw(\hnm)$. By \Prop~\ref{Prop_first_moment} we can choose for any $\alpha>0$ a $\w>0$ large enough such that $\Erw\brk{\Zw}>(1-\alpha)\Erw\brk{Z}$ for large enough $n$. Then \Prop s~\ref{Prop_first_moment} and \ref{Prop_ratio_first_cond} combined with Fact~\ref{Fact_cycles} imply that for any $c_1,...,c_L\le B$ and small enough $\alpha=\alpha(\eps, L, B)$ we have for $n$ large enough:
	\begin{align}\label{eq_lem_conc_2_4}
	\nonumber \Erw\brk{Z|\forall 2\leq l\leq L:C_{l,n}=c_l}& \ge \Erw\brk{\Zw|\forall 2\leq l\leq L:C_{l,n}=c_l}\\
	& \ge \exp\brk{-\eps}\Erw\brk{Z} \prod_{l=2}^L (1+\delta_l)^{c_l}\exp\brk{-\delta_l\lambda_l}.
	\end{align} 
On the other hand, for $\alpha$ sufficiently small and large enough $n$ we have
	\begin{align}\label{eq_lem_conc_2_5}
	 \nonumber  \Erw\brk{Z|\forall 2\leq l\leq L:C_{l,n}=c_l}& = \Erw\brk{Z-Z_\w|\forall 2\leq l\leq L:C_{l,n}=c_l} + \Erw\brk{Z_\w|\forall 2\leq l\leq L:C_{l,n}=c_l}\\
	  \nonumber &\le \frac{2\alpha \Erw\brk Z}{\prod_{l=2}^L\pr\brk{\Po(\lambda_l)=c_l}}+\Erw\brk{Z_\w|\forall 2\leq l\leq L:C_{l,n}=c_l}\\
		    & \le \exp\brk{\eps}\Erw\brk{Z} \prod_{l=2}^L (1+\delta_l)^{c_l}\exp\brk{-\delta_l\lambda_l}
	\end{align}
Thus, the proof of (\ref{eq_lem_conc_2_1}) is completed by combining (\ref{eq_lem_conc_2_3}), (\ref{eq_lem_conc_2_4}), (\ref{eq_lem_conc_2_5}) and taking logarithms. 

\end{proof}

\begin{proof}[Proof of Theorem \ref{Thm_main}]
For $L\ge 2$ we define
\begin{align*}
 W_L=\sum_{l=2}^L X_l \ln(1+\delta_l)-\lambda_l\delta_l.
\end{align*}
Then Fact~\ref{Fact_cycles} implies that for each $L$ the random variables $U_L$ defined in (\ref{eq_lem_conc_2_0}) converge in distribution to $W_L$ as $n\to\infty$. Furthermore, because $\sum_{l}\delta_l \sqrt{\lambda_l}, \sum_{l}\delta_l^2\lambda_l<\infty$, the martingale convergence theorem implies that $W$ is well-defined and that the $W_L$ converge to $W$ almost surely as $L\to\infty$. Therefore, from \Lem s \ref{Lem_conc_2} and \ref{Lem_conc_3} it follows that $\ln Z(\hnm)-\ln\Erw\brk{Z(\hnm)}$ converges to $W$ in distribution, meaning that for any $\eps>0$ we have
\begin{align}\label{eq_proof_main}
\lim_{n\to\infty} \pr\brk{|\ln Z(\hnm)-\ln\Erw\brk{Z(\hnm)}-W|>\eps}=0.
\end{align}
To derive \Thm~\ref{Thm_main} from~(\ref{eq_proof_main}) let $S$ be the event that $\hnm$ consists of $m$ distinct edges. Given that $S$ occurs, $\hnm$ is identical to $H_k(n,m)$.
Furthermore, Fact~\ref{Fact_doubleEdge} implies that $\pr\brk{S}=\Omega(1)$. Consequently, (\ref{eq_proof_main}) yields 
\begin{align}\label{eq_proof_main_1}
\nonumber 0&=\lim_{n\to\infty} \pr\brk{|\ln Z(\hnm)-\ln\Erw\brk{Z(\hnm)}-W|>\eps|S}\\
&=\lim_{n\to\infty} \pr\brk{|\ln Z(H_k(n,m))-\ln\Erw\brk{Z(\hnm)}-W|>\eps}.
\end{align}
Furthermore, \Lem~\ref{Lem_first_moment_balanced} implies that
$\Erw\brk{Z(\hnm)},\Erw\brk{Z(H_k(n,m)}=\Theta\br{2^n\br{1-2^{1-k}}^m}$.
Thus, it holds that $\Erw\brk{Z(\hnm)}=\Theta(\Erw\brk{Z(H_k(n,m)})$ and with (\ref{eq_proof_main_1}) it follows that
$$\lim_{n\to\infty} \pr\brk{|\ln Z(H_k(n,m))-\ln\Erw\brk{Z(H_k(n,m)))}-W|>\eps}=0,$$
which proves \Thm~\ref{Thm_main}.
\end{proof}

\begin{proof}[Proof of \Cor~\ref{Cor_conc}]
The first part of the proof follows directly from \Thm~\ref{Thm_main} and the properties of $W$.
By the definition of convergence in distribution and Markov's inequality we have 
\begin{equation*}
\lim_{n\ra\infty}\pr\brk{|\ln Z(\hknm)-\ln\Erw\brk{Z(\hknm}|\leq\omega}=\pr\brk{|W|\leq\omega} \ge 1-\frac{\Erw|W|}{\omega}
\end{equation*}
and \eqref{eq_Cor_1} follows.

To prove the second part, we construct an event whose probability is bounded away from 0 and that is such that conditioned on this event, the number of solutions of the random hypergraph $\hknm$ is not concentrated very strongly.
\newline
We consider the event $\cT_t$ that the random hypergraph $\hknm$ contains $t$ isolated triangles, i.~e.~$t$ connected components such that each component consists of $3k-3$ vertices and 3 edges and the intersection of each pair of edges contains exactly one vertex.
It is well-known that for $t\ge 0$ there exists $\eps=\eps(d,t)>0$ such that
\begin{align}\label{eq_proof_cor_1}
\liminf_{n \to \infty} \pr\brk{\cT_t}>\eps.
\end{align}
Given $\cT_t$, we let $H_k^*(n,m)$ denote the random hypergraph obtained by choosing a set of $t$ isolated triangles randomly and removing them. Then $H_k^*(n,m)$ is identical to $H_k(n-(3k-3)t,m-3t)$ and with \Prop~\ref{Prop_first_moment} there exists a constant $C=C(d,k)$ such that
\begin{align*}
\Erw\brk{Z(H_k^*(n,m))}=\Erw\brk{Z(H_k(n-(3k-3)t,m-3t))}\le C \cdot 2^{n-(3k-3)t}\br{1-2^{1-k}}^{m-3t}.
\end{align*}
A very accurate calculation of the number of 2-colourings of a triangle in a hypergraph yields that this number is given by
$\br{2^{k-2}-1}\br{2^{2k-1}-2^k+2}$. Thus,  we obtain
\begin{align*}
\Erw\brk{Z(\hknm)|\cT_t}&\le\Erw\brk{Z(H_k(n-(3k-3)t,m-3t))} \br{\br{2^{k-2}-1}\br{2^{2k-1}-2^k+2}}^t\\
& \le C \cdot 2^{n}\br{1-2^{1-k}}^{m-3t} \br{1-2^{2-k}}^t\br{1-2^{1-k}+2^{2-2k}}^t\\
& \le C \cdot 2^{n}\br{1-2^{1-k}}^{m} \br{1-8\br{2^k-2}^{-3}}\\
& \le O\br{\Erw\brk{Z(\hknm)}} \br{1-8\br{2^k-2}^{-3}},
\end{align*}
implying that for any $\omega>0$ we can choose $t$ large enough so that $\Erw\brk{Z(\hknm)|\cT_t}\le \Erw\brk{Z(\hknm)}/(2\exp\brk{\omega})$. Using Markov's inequality, we obtain
\begin{align}\label{eq_proof_cor_2}
\pr\brk{\ln Z(\hknm)\ge \ln \Erw\brk{Z(\hknm)}-\omega|\cT_t}=\pr\brk{Z(\hknm)/\Erw\brk{Z(\hknm)} \ge \exp\brk{-\omega}|\cT_t}\le 1/2.
\end{align}
Thus, combining \eqref{eq_proof_cor_1} and \eqref{eq_proof_cor_2} yields that for any finite $\omega>0$ there is $\eps>0$ such that for large enough $n$ we have
\begin{align*}
\pr\brk{|\ln Z(\hknm)-\Erw\brk{\ln Z(\hknm)}|>\omega}
& \ge \pr\brk{\ln Z(\hknm)<\Erw\brk{\ln Z(\hknm)}-\omega}\\
& \ge \pr\brk{\ln Z(\hknm)\ge \ln \Erw\brk{Z(\hknm)}-\omega|\cT_t} \pr\brk{\cT_t}\\
& >\eps/2,
\end{align*}
thereby completing the proof of the second claim.
\end{proof}

\begin{proof}[Proof of \Cor~\ref{Cor_cont}]
This proof is nearly identical to the one in \cite{aco_plantsil}.
Assume for contradiction that $(\cA_n)_{n\geq1}$ is a sequence of events such that for some fixed number $0<\eps<1/2$ we have
	\begin{equation}\label{eq_cor_cont0}
	\lim_{n\ra\infty}\ppl\brk{\cA_n}=0\quad\mbox{while}\quad\limsup_{n\ra\infty}\prc\brk{\cA_n}>\eps.
	\end{equation}
Let $H_k(n,m,\sigma)$ denote a $k$-uniform hypergraph on $[n]$ with precisely $m$ edges chosen uniformly at random from all edges that are bichromatic under $\sigma$. Let $\cV(\sigma)$ be the event that $\sigma$ is a 2-colouring of $\hknm$.
Then
	\begin{align}\label{eq_cor_cont1}
	\Erw\brk{Z(\hknm)\vecone_{\cA_n}}
	&=\sum_{\sigma:\brk n\ra\{0,1\}} \pr\brk{\cV(\sigma) \text{ and }(H_k(n,m),\sigma)\in\cA_n}\nonumber\\
	&=\sum_{\sigma:\brk n\ra\{0,1\}}\pr\brk{(H_k(n,m),\sigma)\in\cA_n|\cV(\sigma)}\pr\brk{\cV(\sigma)}\nonumber\\
	&=\sum_{\sigma:\brk n\ra\{0,1\}}\pr\brk{H_k(n,m,\sigma)\in\cA_n}\pr\brk{\cV(\sigma)}\nonumber\\
	&\leq O\br{\br{1-2^{1-k}}^m}\sum_{\sigma:\brk n\ra\{0,1\}}\pr\brk{H_k(n,m,\sigma)\in\cA_n}\nonumber\\
	&=O\br{2^n\br{1-2^{1-k}}^m}\pr\brk{H_k(n,m,\SIGMA)\in\cA_n}=o\br{2^n\br{1-2^{1-k}}^m}.
	\end{align}
\noindent
By \Thm~\ref{Thm_main}, for any $\eps>0$ there is $\delta>0$ such that for all large enough $n$ we have
	\begin{equation}\label{eq_cor_cont2}
	\pr\brk{Z(\hknm)<\delta\Erw\brk{Z(\hknm)}}<\eps/2.
	\end{equation}
Now, let $\cE$ be the event that $Z(\hknm)\geq\delta\Erw[Z(\hknm]$ and let $q=\prc\brk{\cA_n|\cE}$.
Then
	\begin{eqnarray}\nonumber
	\Erw\brk{Z(\hknm)\vecone_{\cA_n}}&\geq&\delta\Erw[Z(\hknm)]\pr\brk{((\hknm,\SIGMA)\in\cA_n,\cE}\\
		&\geq&\delta q\Erw[Z(\hknm)]\pr\brk{\cE}\geq\delta q\Erw[Z(\hknm)]/2\nonumber\\
		&=&\frac{\delta q}{2}\cdot\Omega\br{2^n\br{1-2^{1-k}}^m}.\label{eq_cor_cont3}
	\end{eqnarray}
Combining~(\ref{eq_cor_cont1}) and~(\ref{eq_cor_cont3}), we obtain $q=o(1)$.
Hence, (\ref{eq_cor_cont2}) implies that
	\begin{align*}
	\prc\brk{\cA_n}&=\prc\brk{\cA_n|\neg\cE}\cdot\pr\brk{\neg\cE}+q\cdot\pr\brk{\cE}
		\leq\pr\brk{\neg\cE}+q\leq\eps/2+o(1),
	\end{align*}
in contradiction to~(\ref{eq_cor_cont0}).
\end{proof}

\section{The first moment calculation}\label{Sec_first_moment}

The aim in this section is to prove \Prop~\ref{Prop_first_moment} and a result that we need for \Prop~\ref{Prop_ratio_second_first}. 
For a hypergraph $H$ let $Z_{\rho}(H)$ be its number of $2$-colourings with colour density $\rho$. We set $\bar\rho=\frac 12$. For $\rho \in [0,1]$ we define
	\begin{equation}\label{fm_functions}
	f_1:\rho\mapsto \cH(\rho) + g_1(\rho) \qquad \text{with} \quad g_1(\rho)=\frac{d}{k} \ln \br{ 1 - \rho^k-(1-\rho)^k}.
	\end{equation}
The next lemma shows that $f_1(\rho)$ is the function we need to analyse in order to determine the expectation of $Z_\rho$. 	

\begin{lemma} \label{Lem_first_moment_balanced} 
Let $d^{\prime} \in (0, \infty)$. There exist numbers $C_1=C_1(k,d), C_2=C_2(k,d) > 0$ such that for any colour density $\rho$:
		\begin{equation}\label{eq_first_moment_balanced_1}	
		C_1 n^{-1/2} \exp \brk{n f_1(\rho)}\leq \Erw \brk{Z_{\rho}(\hnm)} \leq C_2   \exp \brk{n f_1(\rho)}.
		\end{equation}
Moreover, if $\left| \rho-\bar\rho\right| = o(1)$, then 
		\begin{equation} \label{eq_first_moment_balanced_2}
		\Erw \brk{Z_{\rho}(\hnm)}\sim \sqrt{\frac{2}{\pi n}}\exp \brk{\frac{d(k-1)}{2^{k}-2}}\exp\brk{n f_1(\rho)}.
		\end{equation}
\end{lemma}

\begin{proof}
	The edges in the random hypergraph $\hnm$ are independent by construction, so the expected number of solutions with colour density $\rho$ can be written as
	\begin{equation}\label{eq_first_moment_balanced11}
	\Erw[Z_{\rho}(\hnm)]= { n \choose \rho n }\bc{1-\frac{{\rho n\choose k}+{(1-\rho) n \choose k}}{N}}^m,\quad\mbox{where }N=\bink nk.
	\end{equation}
	Further, the number of ``forbidden'' edges is given by
	\begin{align*} 
	{\rho n\choose k}+{(1-\rho) n \choose k}
	&=\frac{1}{k!}\br{n^k\br{\rho^k+(1-\rho)^k}-\frac{k(k-1)}{2}n^{k-1}\br{\rho^{k-1}+(1-\rho)^{k-1}}+\Theta\br{n^{k-2}}}\\
	&=N\br{\rho^k+(1-\rho)^k}-\frac{k(k-1)}{2k!}n^{k-1}\br{\rho^{k-1}(1-\rho)+\rho(1-\rho)^{k-1}}+\Theta\br{n^{k-2}}
	\end{align*}
	yielding
	\begin{align*}
	1-\frac{{\rho n\choose k}+{(1-\rho) n \choose k}}{N}&=1-\rho^k-(1-\rho)^k+\frac{k(k-1)}{2n}\br{\rho^{k-1}(1-\rho)+\rho(1-\rho)^{k-1}}+\Theta\br{n^{-2}}.
	\end{align*}
	To proceed we observe that $\ln\br{x+\frac{y}{n}}=\ln(x)+\ln\br{1+\frac{y}{xn}}$ for $x>0, y<xn$ and consequently 
	\begin{align}\label{eq_first_moment_balanced12}
	\nonumber &m\ln\br{1-\frac{{\rho n\choose k}+{(1-\rho) n \choose k}}{N}}\\
	\nonumber& \qquad =\frac{dn}{k}\br{\ln\br{1-\rho^k-(1-\rho)^k}+\ln\br{1-\frac{k(k-1)}{2n}\frac{\rho^{k-1}(1-\rho)+\rho(1-\rho)^{k-1}}{1-\rho^k-(1-\rho)^k}+\Theta\br{n^{-2}}}}\\
	&\qquad \sim \frac{dn}{k}\ln\br{1-\rho^k-(1-\rho)^k}+\frac{d(k-1)}{2}\br{\frac{\rho^{k-1}(1-\rho)+\rho(1-\rho)^{k-1}}{1-\rho^k-(1-\rho)^k}}+\Theta\br{n^{-1}}.
	\end{align}
	Equation (\ref{eq_first_moment_balanced_1}) follows from (\ref{eq_first_moment_balanced11}), (\ref{eq_first_moment_balanced12}) and Stirling's formula applied to ${ n \choose \rho n }$.
	Moreover, equation (\ref{eq_first_moment_balanced_2}) follows from~(\ref{eq_first_moment_balanced11}) 
	and~(\ref{eq_first_moment_balanced12}) because $\left|\rho-\bar\rho\right|=o(1)$ implies that	
	\begin{equation*} 
	{ n \choose \rho n} \sim \sqrt{\frac{2}{\pi n}} \exp \brk{n \cH(\rho)} \qquad \text{and} \qquad \frac{\rho^{k-1}(1-\rho)+\rho(1-\rho)^{k-1}}{1-\rho^k-(1-\rho)^k}\sim \frac{1}{2^{k-1}-1}.
	\end{equation*}
\end{proof}
\noindent The following corollary states an expression for $\Erw \brk{Z(\hnm)}$. Additionally, it shows that when $\omega\ra\infty$, this value can be approximated by $\Erw \brk{\Zw(\hnm)}$.

\begin{corollary} \label{Cor_first_moment_total}
	Let $d^{\prime} \in (0, \infty)$. Then 
	\begin{align}\label{eq_fm_total_1}
	\Erw \brk{Z(\hnm)} \sim \exp \brk{\frac{d(k-1)}{2^{k}-2}+nf_1(\bar\rho)}\br{1+\frac{d(k-1)}{2^{k-1}-1}}^{-\frac 12}.
	\end{align}
	Furthermore, for $\w>0$  we have 
	\begin{align}\label{eq_fm_total_2}
	\lim_{\w\to\infty} \lim_{n \to \infty} \frac{\Erw\brk{\Zw(\hnm)}}{\Erw \brk{Z(\hnm)}}=1.
	\end{align}
\end{corollary}

\begin{proof}
	The functions  $\rho \mapsto \cH(\rho)$ and $\rho \mapsto g_1(\rho)$ are both concave and attain their maximum at $\rho=\bar\rho$. Consequently, setting $B(d,k) = 4\br{1+\frac{d(k-1)}{2^{k-1}-1}}$ and expanding around $\bar \rho$, we obtain		
	\begin{equation}\label{eq_first_moment total_1}
	f_1(\bar\rho) - \frac{B(d,k)}{2} \br{\rho-\bar{\rho}}^2 - O\br{\br{\rho-\bar{\rho}}^3} \leq  f_1(\rho) \leq f_1(\bar\rho) - \frac{B(d,k)}{2} \br{\rho-\bar{\rho}}^2. 
	\end{equation}
	Plugging the upper bound from~(\ref{eq_first_moment total_1}) into (\ref{eq_first_moment_balanced_1}) and observing that the number of all colour densities for maps $\sigma:[n] \to \{0,1\}$ is bounded from above by $n=\exp[o(n)]$, we find		
	\begin{equation}\label{eq_first_moment total_2}
	S_1=\sum_{\rho:\ |\rho - \bar \rho|> n^{-3/8} } 
	\Erw \brk{Z_{\rho}(\hnm)} \leq C_2 \exp \brk{n f_1(\bar\rho)- \frac{B(d,k)}{2}n^{1/4}}.
	\end{equation}
	On the other hand, equation (\ref{eq_first_moment_balanced_2}) implies that
	\begin{align} \label{eq_first_moment total_3}
	S_2&=\sum_{\rho:\ |\rho - \bar \rho| \leq n^{-3/8}} \Erw \brk{Z_{\rho}(\hnm)} 
	\sim \sqrt{\frac{2}{\pi n}}\exp \brk{\frac{d(k-1)}{2^{k}-2}}+\exp\brk{n f_1(\bar\rho)} \sum_{\rho} \exp \brk{- n \frac{B(d,k)}{2} \br{\rho - \bar \rho}^2}. 
	\end{align}	
	The last sum is in the standard form of a Gaussian summation. Using $\int_{-\infty}^{\infty} \exp\brk{-a(x+b)^2}dx=\sqrt{\frac a\pi}$ we get  
	\begin{align}\label{eq_first_moment total_4}
	\nonumber \sum_{\rho \in \cA(n)} \exp \brk{- n \frac{B(d,k)}{2}\br{\rho - \bar \rho}^2} \nonumber &\sim n \int \exp \brk{- n \frac{B(d,k)}{2} \br{\rho - \bar \rho}^2} d\rho\\
	&\sim n\sqrt{\frac{2\pi}{nB(d,k)}} = \sqrt{\frac{\pi n}{2}}  \br{1+\frac{d(k-1)}{2^{k-1}-1}}^{-\frac 12}
	\end{align}
	Plugging~(\ref{eq_first_moment total_4}) into (\ref{eq_first_moment total_3}), we obtain
	\begin{align}\label{eq_first_moment total_5}
	S_2	&\sim \exp \brk{\frac{d(k-1)}{2^{k}-2}+nf_1(\bar\rho)}\br{1+\frac{d(k-1)}{2^{k-1}-1}}^{-\frac 12}.
	\end{align}
	Finally, comparing~(\ref{eq_first_moment total_2}) and~(\ref{eq_first_moment total_5}), we see that $S_1=o(S_2)$. Thus, $S_1+S_2\sim S_2$ and (\ref{eq_fm_total_1}) follows from~(\ref{eq_first_moment total_5}).
	
	\noindent To prove (\ref{eq_fm_total_2}), we find that analogously to (\ref{eq_first_moment total_2}), (\ref{eq_first_moment total_3}) and the calculation leading to (\ref{eq_first_moment total_5}), it holds that	\begin{equation*}
	S_1'=\sum_{\rho:\ |\rho - \bar \rho|> \w n^{-1/2} } \Erw \brk{Z_{\rho}(\hnm)} \leq C_2 \exp \brk{n f(\bar\rho)- \frac{B(d,k)}{2}\w}.
	\end{equation*}
	and 	
	\begin{align*} 
	S_2' =\sum_{\rho:\ |\rho - \bar \rho| \leq \w n^{-1/2}} \Erw \brk{Z_{\rho}(\hnm)} \sim \exp \brk{\frac{d(k-1)}{2^{k}-2}+nf_1(\bar\rho)}\br{1+\frac{d(k-1)}{2^{k-1}-1}}^{-\frac 12}. 
	\end{align*}	
	Thus, we have $\lim_{\w\ra\infty} \lim_{n \to \infty} \frac{S_1'+S_2'}{S_2'}=1$, yielding (\ref{eq_fm_total_2}).
\end{proof}

\begin{proof}[Proof of \Prop~\ref{Prop_first_moment}]
	The statements are immediate by \Cor~\ref{Cor_first_moment_total} and the fact that $$f_1\br{\bar\rho}=\ln 2+\frac dk\ln\br{1-2^{1-k}}.$$
\end{proof}

\noindent Finally, we derive an expression for $\Erw\brk{\Zwni(\hnm)}$ that we will need to prove \Prop~\ref{Prop_ratio_second_first}.
\begin{lemma}\label{Lem_fm_shifted}
 Let $d^{\prime} \in (0,\infty), \w, \nu \in \mathbb N, s\in[\w\nu]$ and $\rho\in\Awni$. Then with $\rwni$ as defined in (\ref{rwni}) we have
\begin{align*}
 \Erw\brk{\Zwni(\hnm)}\sim_{\nu} |\Awni|\sqrt{\frac{2}{\pi n}}\exp \brk{\frac{d(k-1)}{2^{k}-2}}\exp\brk{n f_1\br{\rwni}}.
\end{align*}
\end{lemma}

\begin{proof}
Using a Taylor expansion of $f_1(\rho)$ around $\rho=\rwni$, we get
 \begin{align}\label{eq_fm_shifted}
 f_1(\rho)=f_1(\rwni)+\Theta\br{\frac \w{\sqrt n}}|\rho-\rwni|+\Theta\br{\br{\rho-\rwni}^2}.
 \end{align}
As $|\rho-\rwni|\le \frac{1}{\nu \sqrt n}$ for $\rho\in\Awni$, we conclude that  $f_1(\rho)=f_1(\rwni)+O\br{\frac \w{\nu n}}$ and as this is independent of $\rho$ the assertion follows by inserting (\ref{eq_fm_shifted}) in (\ref{eq_first_moment_balanced_2}) and multiplying with $|\Awni|$.
\end{proof}

\section{Counting short cycles}\label{Sec_short_cycles}

We recall that for $l \in \{2,\ldots,L\}$ we denote by $C_{l,n}$ the number of cycles of length $l$ in $\hnm$. Further we let $c_2,\ldots,c_L$ be a sequence of non-negative integers and $S$ be the event that $C_{l,n}=c_l$ for $l=2,\ldots,L$. Additionally, for an assignment $\sigma:[n] \to \{0,1\}$ we let $\cV(\sigma)$ be the event that $\sigma$ is a colouring of the random graph $\hnm$.
We also recall $\lambda_l,\delta_l$ from~(\ref{eq_lambdadelta}).

\begin{proof}[Proof of \Prop~\ref{Prop_ratio_first_cond}]
First observe that from the definition of $\lambda_l$ and $\delta_l$ in (\ref{eq_lambdadelta}) and the fact that $\sum_{n=1}^\infty\frac{x^n}n=-\ln(1-x)$ we get
\begin{align}\label{eq_lambdadeltainfty}
\exp \brk{\sum_{l \geq 2} \lambda_l \delta_l^2} = \exp\brk{-\frac{d(k-1)}2\frac 1{\br{2^{k-1}-1}^2}}\br{1-\frac{d(k-1)}{\br{2^{k-1}-1}^2}}^{-1/2}.
\end{align}
Together with \eqref{eq_lambdadeltainfty},
\Prop~\ref{Prop_ratio_first_cond} readily follows from the following lemma about the distribution of the random variables $C_{l,n}$ given $\cV(\sigma)$.

\begin{lemma}\label{Lem_planted_cycles}
	Let $\mu_l=\frac{(d(k-1))^l}{2l}\brk{1+\frac{(-1)^l}{\br{2^{k-1}-1}^l}}$.
	Then
	$\pr\brk{S|\cV(\sigma)} \sim \prod_{l=2}^{L}\frac{\exp\brk{-\mu_l}}{c_l!}\mu_l^{c_l}$ for any $\sigma$ with $\rho(\sigma)\in\Aw$.
\end{lemma}

\noindent
Before we establish \Lem~\ref{Lem_planted_cycles}, let us point out how it implies \Prop~\ref{Prop_ratio_first_cond}. By Bayes' rule, we have
\begin{align}\label{eq_cycles_1}
\Erw\brk{\Zwni(\hnm)| S}&= \frac{1}{\pr[S]}\sum_{\tau\in\Awni}\pr[\cV(\tau)] \pr[S|\cV(\tau)].
\end{align}
Inserting the result from \Lem~\ref{Lem_planted_cycles} into (\ref{eq_cycles_1}) yields
\begin{align*}
\Erw\brk{\Zwni(\hnm)| S}
&\sim \frac{\prod_{l=2}^{L}\frac{\exp\brk{-\mu_l}}{c_l!}\mu_l^{c_l}}{\pr[S]}\sum_{\tau\in \Awni}\pr[\cV(\tau)]
\sim \frac{\prod_{l=2}^{L}\frac{\exp\brk{-\mu_l}}{c_l!}\mu_l^{c_l}}{\pr[S]}\Erw\brk{\Zwni(\hnm)}.  
\end{align*}
From \Lem~\ref{Lem_planted_cycles} and Fact~\ref{Fact_cycles} we get that
\begin{align*}
\frac{\prod_{l=2}^{L}\frac{\exp\brk{-\mu_l}}{c_l!}\mu_l^{c_l}}{\pr[S]}\sim\prod_{l=2}^L\left[1+\delta_l\right]^{c_l}\exp\brk{-\delta_l\lambda_l}
\end{align*}
and \Prop~\ref{Prop_ratio_first_cond} follows.
\end{proof}

\begin{proof}[Proof of \Lem~\ref{Lem_planted_cycles}]
We are going to show that for any fixed sequence of integers $m_1, \ldots, m_L\geq0$, the joint factorial moments satisfy 
\begin{eqnarray}
\mathbb{E}\brk{(C_{2,n})_{m_2}\cdots (C_{L,n})_{m_L}|\cV(\sigma)}\sim
\prod_{l=2}^L \mu_l^{m_l}. \label{eq_cycles_joint_moments}
\end{eqnarray}
Then \Lem~\ref{Lem_planted_cycles} follows from~\cite[\Thm~1.23]{Bollobas}.

We  consider the number of sequences of $m_2+\cdots+m_L$ distinct cycles such that 
 $m_2$ corresponds to the number of cycles of length $2$, and so on. Clearly this number is equal
to $(C_{2,n})_{m_2}\cdots (C_{L,n})_{m_L}$.  

We call a cycle \textit{good}, if it does not contain edges that overlap on more than one vertex. We call a sequence of good cycles \textit{good sequence} if for any two cycles $C$ and $C'$ in this sequence, there are no vertices $v \in C$ and $v' \in C'$ such that $v$ and $v'$ are contained in the same edge. Let $Y$ be the number of good sequences and $\bar Y$ be the number of sequences that are not good. Then it holds that
\begin{eqnarray}\label{eq_SplitToDisjointJointCase}
\mathbb{E}\brk{(C_{2,n})_{m_2}\cdots (C_{L,n})_{m_L}| \cV(\sigma)}=\mathbb{E}[Y|\cV(\sigma)]+\mathbb{E}[\bar Y|\cV(\sigma)].
\end{eqnarray}
The following claim states that the contribution of $\mathbb{E}[\bar Y|\cV(\sigma)]$ is negligible. Its proof follows at the end of this section. 

\begin{claim}\label{Claim_overlapCycles}
We have $\mathbb{E}\brk{\bar Y|\cV(\sigma)}=O\br{n^{-1}}$.
\end{claim}
\noindent

Thus it remains to count good sequences given $\cV(\sigma)$. We let $\sigma \in \Aw$ and first consider the number $D_{l,n}$ of rooted, directed, good cycles of length $l$. This will introduce a factor of $2l$ for the number of all good cycles of length $l$, thus $D_{l,n}=2lC_{l,n}$. For a rooted, directed, good cycle of length $l$ we need to pick $l$ vertices $(v_1,...,v_l)$ as roots, introducing a factor $(1+o(1))\br{\frac{n}{2}}^l$, and there have to exist edges between them which generates a factor $\brk{\frac{m}{\binom{n}{k}(1-2^{1-k})}}^l$. To choose the remaining vertices in the participating edges we have to distinguish between pairs of vertices $(v_i,v_{i+1})$ that are assigned the same colour and those that are not, because if $\sigma(v_i)=\sigma(v_{i+1})$ we have to make sure that at least one of the other $k-2$ vertices participating in this edge is assigned the opposite colour. This gives rise to the third factor in the following calculation. 

\begin{align}
\mathbb{E}\brk{D_{l,n} |\cV(\sigma)}
& \nonumber \sim \br{\frac{n}{2}}^l\brk{\frac{m}{\binom{n}{k}\br{1-2^{1-k}}}}^l \cdot 2\sum_{i=0}^l \brk{\binom{l}{i}\binom{n-2}{k-2}^i\brk{\binom{n-2}{k-2}-\binom{n/2}{k-2}}^{l-i}\mathrm 1_{\{i \text{ is even}\}}}\\
& \nonumber = \br{\frac{n}{2}}^l\brk{\frac{m}{\binom{n}{k}\br{1-2^{1-k}}}}^l \cdot \brk{\brk{2\binom{n-2}{k-2}-\binom{n/2}{k-2}}^l+\brk{-\binom{n/2}{k-2}}^l}\\
& \nonumber \sim \br{\frac{n}{2}}^l\brk{\frac{k!dn}{kn^k\br{1-2^{1-k}}}}^l \cdot \brk{\frac{\brk{\br{2^{k-1}-1}n^{k-2}}^l+\br{-n^{k-2}}^l}{\brk{2^{k-2}(k-2)!}^l}}\\
& \nonumber =\brk{d(k-1)}^l \br{1+\frac{(-1)^l}{\br{2^{k-1}-1}^l}}
\end{align}
Hence, recalling that $C_{l,n}=\frac1{2l}D_{l,n}$, we get
\begin{align}\label{eq_number_cycles}
\Erw\brk{C_{l,n} |\cV(\sigma)}&\sim \frac{\brk{d(k-1)}^l}{2l} \br{1+\frac{(-1)^l}{\br{2^{k-1}-1}^l}}. 
\end{align}
In fact, since $Y$ considers only good sequences and $l$, $m_2,\ldots, m_L$ remain fixed as $n\ra\infty$, (\ref{eq_number_cycles}) yields
\begin{eqnarray}
\mathbb{E}[Y|\cV(\sigma)] \sim\prod_{l=2}^L \br{\frac{\brk{d(k-1)}^l}{2l} \br{1+\frac{(-1)^l}{\br{2^{k-1}-1}^l}}}^{m_l} \nonumber.
\end{eqnarray}
Plugging the above relation and Claim  \ref{Claim_overlapCycles} into (\ref{eq_SplitToDisjointJointCase}) we get 
(\ref{eq_cycles_joint_moments}). The proposition follows.
\end{proof}

\begin{claimproof}{ \ref{Claim_overlapCycles}}
The idea of the proof is to find an event, namely that there exists an induced subgraph with too many edges, that always occurs if $\bar Y>0$ and whose probability we can bound from above. To this aim let $A=\{i\in \mathbb R|i=(l-1)(k-1)+j \text{ for some } l\le L, j\in\{0,...,k-2\}\}$. For every subset $R$ of $(l-1)(k-1)+j$ vertices, where $l\leq L$ and $j \in \{0,...,k-2\}$  let  $\mathbb{I}_{R}$ be equal to 1 if the number of edges that only consist of vertices in $R$ is at least $l$. 
Let the $H_L$ be the event that $\sum_{R:|R|\in A}\mathbb{I}_{R}>0$.  It is direct to check that if $\bar Y>0$ then $H_L$ occurs. This implies that
\begin{eqnarray}
\pr\brk{\bar Y>0|\cV(\sigma)}\leq \pr\brk{H_L|\cV(\sigma)}. \nonumber
\end{eqnarray}
The claim follows by appropriately bounding $\pr\brk{H_L|\cV(\sigma)}$.
For this we are going to use Markov's inequality, i.e. 
\begin{eqnarray}
\pr\brk{H_L|\cV(\sigma)}&\leq& \mathbb{E}\brk{\sum_{R:|R|\in A}\mathbb{I}_{R}|\cV(\sigma)} \nonumber
	=\sum_{l=2}^L\sum_{j=0}^{k-2}\sum_{R:|R|=(l-1)(k-1)+j}\mathbb{E}\brk{\mathbb{I}_{R}|\cV(\sigma)}\nonumber.
\end{eqnarray}
For any set $R$ such that $|R|=(l-1)(k-1)+j$, 
we can put $l$ edges inside  the set in at most $\br{\substack{\binom{(l-1)(k-1)+j}{k}\\l}}$ ways, which obviously gets largest if $j=k-2$ and thus $(l-1)(k-1)+j=l(k-1)-1$. Clearly conditioning on $\cV(\sigma)$ can only reduce the  number of different placings of the edges.

We observe that for a colouring $\sigma$ and two fixed vertices $v$ and $v'$ with $\sigma(v)\ne \sigma(v')$ the probability that $e(v,v')$ does not exist is $\br{1-\frac{1}{N-\cF\bc{\sigma}}}^m$.
Using inclusion/exclusion and the binomial theorem, with $N=\binom{n}{k}$ and $\cF\bc{\sigma}\sim 2^{1-k}N$, for a fixed set $R$ of cardinality $(l-1)(k-1)+j$ we get that
\begin{eqnarray}
\mathbb{E}\brk{\mathbb{I}_{R}|\cV(\sigma)}&\leq& \br{\substack{\br{\substack{l(k-1)-1\\k}}\\l}}  
{\sum_{i=0}^{l}{l \choose i}(-1)^i\br{1-\frac{i}{N-\cF\bc{\sigma}}}^m}
\nonumber\\
&\leq& \nonumber \br{\substack{\br{\substack{l(k-1)-1\\k}}\\l}} \br{\frac{m}{N-\cF\bc{\sigma}}}^{l} \sim \br{\substack{\br{\substack{l(k-1)-1\\k}}\\l}} \br{\frac{m}{\binom{n}{k}\br{1-2^{1-k}}}}^{l}. 
\nonumber
\end{eqnarray}
With $m=\frac{dn}{k}$ and since ${i\choose j}\leq \br{ie/j}^j$, it holds that  
\begin{align*}
\pr\brk{H_L|\cV(\sigma)} &\leq (1+o(1))\sum_{l=2}^L{n \choose l(k-1)-1}\br{\substack{\br{\substack{l(k-1)-1\\k}}\\l}} \br{\frac{m}{\binom{n}{k}\br{1-2^{1-k}}}}^{l} \nonumber\\
&=(1+o(1)) \sum_{l=2}^L\br{\frac{ne}{l(k-1)-1}}^{l(k-1)-1}\br{\frac{e^{k+1}(l(k-1)-1)^k}{k^kl}}^{l} \br{\frac{mk^k}{n^ke^k\br{1-2^{1-k}}}}^{l} \nonumber\\
&=(1+o(1)) \sum_{l=2}^L \frac{m^le^{kl-1}(l(k-1)-1)^{l+1}}{n^{l+1}l^l\br{1-2^{1-k}}^l} \nonumber \\
&=\frac{1+o(1)}{n}  \sum_{l=2}^L \br{\frac{e^kd(l(k-1)-1)}{l\br{1-2^{1-k}}}}^l\frac{l(k-1)-1}{e}=O\br{n^{-1}},
\end{align*}
where the last equality follows since $L$ is a fixed number.
\end{claimproof}

\section{The second moment calculation}\label{Sec_second_moment}

\noindent
In this section we prove \Prop~\ref{Prop_ratio_second_first}. To this end, we need to derive an expression for the second moment of the random variables $\Zwni$ for $s\in[\w\nu]$ that is asymptotically tight. As a consequence, we need to put more effort into the calculations than done in prior work on hypergraph-2-colouring (e.g.\cite{Lenka}), where the second moment of $Z$ is only determined up to a constant factor. Part of the proof is based on ideas from \cite{aco_plantsil}, but as we aim for a stronger result, the arguments are extended and adapted to our situation. 

\subsection{The overlap}
For two colour assignments $\sigma, \tau : [n] \to \{0,1\}$ we define the \textit{overlap matrix} 
$$\rho(\sigma, \tau)=
	\begin{pmatrix}
	\rho_{00}(\sigma, \tau) & \rho_{01}(\sigma, \tau) \\ \rho_{10}(\sigma, \tau) & \rho_{11}(\sigma, \tau)
	\end{pmatrix}$$
with entries
	$$ \rho_{ij}(\sigma, \tau) = \frac{1}{n} \cdot | \sigma^{-1}(i) \cap \tau^{-1}(j) | \quad \text{for } i,j\in\{0,1\}.$$
Obviously, it holds that 
$$\rho_{00}(\sigma, \tau) + \rho_{01}(\sigma, \tau) + \rho_{10}(\sigma, \tau) + \rho_{11}(\sigma, \tau)=1.$$
If we further remember the definition from (\ref{rhosigma}), we can alternatively represent $\rho(\sigma,\tau)$ as
$$\rho(\sigma, \tau)=
	\begin{pmatrix}
	\rho_{00}(\sigma, \tau) & \rho(\sigma)-\rho_{00}(\sigma, \tau) \\ \rho(\tau)-\rho_{00}(\sigma, \tau) & 1-\rho(\sigma)-\rho(\tau)+\rho_{00}(\sigma, \tau)
	\end{pmatrix}.$$
To simplify the notation, for a $2\times 2$-matrix $\rho=(\rho_{ij})$ we introduce the shorthands
	$$ \rho_{i \star} = \rho_{i0}+\rho_{i1}, \qquad\rho_{\nix\star}=(\rho_{0 \star}, \rho_{1 \star}),
 \qquad \qquad \rho_{\star j} = \rho_{0j}+\rho_{1j},\qquad \rho_{\star\nix}=(\rho_{\star 0}, \rho_{\star 1}). $$
We let $\cB(n)$ be the set of all overlap matrices $\rho(\sigma,\tau)$ for $\sigma,\tau:[n]\to\{0,1\}$ and $\cB$ denote the set of all probability distributions $\rho=(\rho_{ij})_{i,j\in\{0,1\}}$ on $\{0,1\}^2$. Further, we let $\bar\rho$ be the $2\times 2$-matrix with all entries equal to $1/4$.

\noindent For a given hypergraph $H$ on $[n]$, let $Z^{(2)}_{\rho}(H)$ be the number of pairs $(\sigma,\tau)$ of $2$-colourings of $H$ whose overlap matrix is $\rho$. Analogously to (\ref{fm_functions}), we define the functions $f_2, g_2:\cB\mapsto \mathbb R$ as
	\begin{equation*}
	f_2:\rho\mapsto \cH(\rho) + g_2(\rho) \qquad \text{with} \quad g_2(\rho)=\frac{d}{k} \ln \br{1-\sum \rho_{i\star}^k-\sum \rho_{\star j}^k+\sum \rho_{ij}^k}.
	\end{equation*}
The following lemma states a formula for $\Erw\brk{Z^{(2)}_{\rho}(\hnm)}$ for $\rho\in\cB(n)$ in terms of $f_2(\rho)$. 

\begin{lemma}\label{Lem_sm_rho}
	Let $d^{\prime} \in (0,\infty)$ and set
	\begin{align}\label{eq_sm_rho}
	C_n(d,k)=\sqrt{\frac{32}{(\pi n)^3}}\exp\brk{\frac{d(k-1)}2\frac{2^k-3}{\br{2^{k-1}-1}^2}}.
	\end{align}
	Then for $\rho \in \cB(n)$ we have
		\begin{align}\label{eq_sm_rho_1}
		\Erw \brk{Z_{\rho}^{(2)}(\hnm)}
		\nonumber	&\,
		\sim \sqrt{\frac{2\pi}{n^3}}\prod_{i,j=1}^2(2\pi \rho_{ij})^{-1/2}\exp\brk{nf_2(\rho)}\\
		&\quad \exp\brk{\frac{d(k-1)}2\frac{\sum \rho_{i\star}^{k-1}-\sum \rho_{i\star}^{k}+\sum \rho_{\star j}^{k-1}-\sum \rho_{\star j}^{k}-\sum \rho_{ij}^{k-1}+\sum \rho_{ij}^{k}}{1-\sum \rho_{i\star}^k-\sum \rho_{\star j}^k+\sum \rho_{ij}^k}}.
		\end{align}
	Moreover, if  $\rho \in \cB(n)$ satisfies $\left\|\rho-\bar\rho\right\|_2^2=o(1)$, then
		\begin{align}\label{eq_sm_rho_2}
		\Erw \brk{Z_{\rho}^{(2)}(\hnm)} \sim C_n(d,k)\exp\brk{nf_2(\rho)}.
		\end{align}
\end{lemma}

\begin{proof}
	Let  $\rho=
	\begin{pmatrix}
	\rho_{00} & \rho_{10} \\ \rho_{01} & \rho_{11}
	\end{pmatrix} \in \cB(n)$. Then
	\begin{align}\label{eq_sm_estimates_complete}
	\nonumber\Erw \brk{Z_{\rho}^{(2)}(\hnm)}&=\sum_{\sigma, \tau: \rho(\sigma, \tau)=\rho}\pr\brk{\sigma, \tau \text{ are colourings of }\hnm}=\sum_{\sigma, \tau: \rho(\sigma, \tau)=\rho} \br{1-\frac{\mathcal F(\sigma,\tau)}{N}}^m\\
	&={n \choose \rho_{00}n,\rho_{01}n, \rho_{10}n, \rho_{11}n}\br{1-\frac{\mathcal F(\sigma,\tau)}{N}}^m.
	\end{align}
	where $N={n\choose k}$ and $\mathcal F(\sigma,\tau)$ is the total number of possible monochromatic edges under either $\sigma$ or $\tau$. In the last line, $\sigma$ and $\tau$ are just two arbitrary fixed $2$-colourings with overlap $\rho$ and the equation is valid because the following computation shows that $\mathcal F(\sigma,\tau)$ only depends on $\rho$:	
	\begin{align*}
	\nonumber \mathcal F(\sigma,\tau)&= \sum_{i=0}^1 {\rho_{i\star}n \choose k}+\sum_{j=0}^1 {\rho_{\star j}n \choose k}-\sum_{i,j=0}^1 {\rho_{ij}n \choose k}\\
	\nonumber &=N\brk{\sum_{i=0}^1 \rho_{i\star}^k+\sum_{j=0}^1 \rho_{\star j}^k-\sum_{i,j=0}^1 \rho_{ij}^k}+\frac{k(k-1)}{2k!}n^{k-1}\cdot \\
	&\qquad \brk{\sum_{i=0}^1 \rho_{i\star}^k-\sum_{i=0}^1 \rho_{i\star}^{k-1}+\sum_{j=0}^1 \rho_{\star j}^k-\sum_{j=0}^1 \rho_{\star j}^{k-1}-\sum_{i,j=0}^1  \rho_{ij}^k+\sum_{i,j=0}^1  \rho_{ij}^{k-1}}+\Theta\br{n^{k-2}},
	\end{align*}
	yielding
	\begin{align*}
	\nonumber 1-\frac{F(\sigma,\tau)}{N}	&=1-\sum_{i=0}^1 \rho_{i\star}^k-\sum_{j=0}^1 \rho_{\star j}^k+\sum_{i,j=0}^1 \rho_{ij}^k\\
	&\qquad -\frac{k(k-1)}{2n} \brk{\sum_{i=0}^1 \rho_{i\star}^k-\sum_{i=0}^1 \rho_{i\star}^{k-1}+\sum_{j=0}^1 \rho_{\star j}^k-\sum_{j=0}^1 \rho_{\star j}^{k-1}-\sum_{i,j=0}^1  \rho_{ij}^k+\sum_{i,j=0}^1  \rho_{ij}^{k-1}}+\Theta\br{n^{-2}}.
	\end{align*}
	We proceed as in the proof of \Lem~\ref{Lem_first_moment_balanced} by using that $\ln\br{x-\frac{y}{n}}=\ln(x)+\ln\br{1-\frac{y}{xn}}$ for $x>0, \frac yn<x$ and consequently 
	\begin{align}\label{eq_sm_estimates_1}
	\nonumber m\ln\br{1-\frac{F(\sigma,\tau)}{N}} 
	& =\frac{dn}{k}\left[\ln\br{1-\sum \rho_{i\star}^k-\sum \rho_{\star j}^k+\sum \rho_{ij}^k}\right.\\
	&\nonumber \left. +\ln\br{1-\frac{k(k-1)}{2n}\frac{\sum \rho_{i\star}^{k}-\sum \rho_{i\star}^{k-1}+\sum \rho_{\star j}^{k}-\sum \rho_{\star j}^{k-1}-\sum \rho_{ij}^{k}+\sum \rho_{ij}^{k-1}}{1-\sum \rho_{i\star}^k-\sum \rho_{\star j}^k+\sum \rho_{ij}^k}+\Theta\br{n^{-2}}}\right]\\
	&\nonumber \sim \frac{dn}{k}\ln\br{1-\sum \rho_{i\star}^k-\sum \rho_{\star j}^k+\sum \rho_{ij}^k}\\
	&+\frac{d(k-1)}{2}\frac{\sum \rho_{i\star}^{k-1}-\sum \rho_{i\star}^{k}+\sum \rho_{\star j}^{k-1}-\sum \rho_{\star j}^{k}-\sum \rho_{ij}^{k-1}+\sum \rho_{ij}^{k}}{1-\sum \rho_{i\star}^k-\sum \rho_{\star j}^k+\sum \rho_{ij}^k}+\Theta\br{n^{-1}}.
	\end{align}
	As $\mathcal F(\sigma,\tau)$ does only depend on $\rho$, \eqref{eq_sm_estimates_complete} becomes
	Using Stirling's formula, we get the following approximation for the number of colour assignments with overlap $\rho$: 
	\begin{align}\label{eq_sm_estimates_2}
	{n \choose \rho_{00}n,\rho_{01}n, \rho_{10}n, \rho_{11}n}&\sim \sqrt{2\pi}n^{-3/2}\prod_{i,j=0}^1(2\pi \rho_{ij})^{-1/2}\exp\brk{n\cH(\rho)}.
	\end{align}
        Inserting (\ref{eq_sm_estimates_1}) and (\ref{eq_sm_estimates_2}) into (\ref{eq_sm_estimates_complete}) completes the proof of (\ref{eq_sm_rho_1}).
    
    Equation (\ref{eq_sm_rho_2}) follows from (\ref{eq_sm_rho_1}) because if $\left\|\rho-\bar\rho\right\|_2^2=o(1)$ then
     	\begin{align*}
     	\prod_{i,j=1}^2(2\pi \rho_{ij})^{-1/2}\sim \frac{4}{\pi^2} \quad \text{and} \quad \frac{\sum \rho_{i\star}^{k-1}-\sum \rho_{i\star}^{k}+\sum \rho_{\star j}^{k-1}-\sum \rho_{\star j}^{k}-\sum \rho_{ij}^{k-1}+\sum \rho_{ij}^{k}}{1-\sum \rho_{i\star}^k-\sum \rho_{\star j}^k+\sum \rho_{ij}^k}\sim \frac{2^k-3}{\br{2^{k-1}-1}^2}.
     	\end{align*}
     
\end{proof}

\subsection{Dividing up the interval}
Let $\w,\nu \in \mathbb N$ and $s\in[\w\nu]$. Analogously to the notation in \Sec~\ref{Sec_outline} we introduce the sets
	\begin{align*}
	\Bw&=\cbc{\rho\in\cB(n): \rho_{i \star}, \rho_{\star i} \in \left[\frac 12-\frac{\w}{\sqrt{n}},\frac 12+\frac{\w}{\sqrt{n}}\right)\quad \mbox{ for $i\in \{0,1\}$}}
	\end{align*}
and
	\begin{align*}
	\Bwni&=\cbc{\rho\in\Bw: \rho_{i \star}, \rho_{\star i} \in \left[\rwni-\frac{1}{\nu\sqrt{n}},\rwni+\frac{1}{\nu\sqrt{n}}\right)\quad \mbox{ for $i\in \{0,1\}$}},
	\end{align*}
imposing constraints on the overlap matrix $\rho$ insofar as the colour densities resulting from its projection on each colouring must not deviate too much from $1/2$ in the set $\Bw$ and from $\rwni$ in the set $\Bwni$. By the linearity of expectation, for any $s\in[\w\nu]$ we have
	\begin{align*}
	\Erw \brk{\Zwni(\hnm)^2}&=
		\sum_{\rho\in\Bwni}\Erw\brk{Z^{(2)}_{\rho}(\hnm)}.
	\end{align*}
We are going to show that the expression on the right hand side of this equation is dominated by the contributions with $\rho$ ``close to'' $\bar\rho$ in terms of the euclidian norm. More precisely, for $\eta>0$ we introduce the set
	\begin{align*}
	\Bwnie&=\cbc{\rho\in\Bwni:\norm{\rho-\bar\rho}_2\leq\eta}
	\end{align*}
and define 
	$$\Zwnie(\hnm)=\sum_{\rho\in\Bwnie}Z^{(2)}_{\rho}(\hnm).$$
The following proposition reveals that it suffices to consider overlap matrices $\rho$ such that
	$\norm{\rho-\bar\rho}_2\leq n^{-3/8}$.
Here, the number $3/8$ is somewhat arbitrary, any number smaller than $1/2$ would do.



 

\begin{proposition} \label{Prop_sm_eta}
Let $k \ge 3$ and $\w,\nu \in \mathbb N$. If $d^{\prime}/k< 2^{k-1}\ln 2-2$, than for every $s\in\brk{\w\nu}$ we have $$\Erw \brk{\Zwni(\hnm)^2} \sim \Erw \brk{\Zwnis(\hnm)}. $$
\end{proposition}

\noindent To prove this proposition, we need the following lemma. 

\begin{lemma}\label{Lem_sm_norm}
 Let $d/k< 2^{k-1}\ln 2-2$ and $C_n(d,k)$ as defined in \Lem~\ref{Lem_sm_rho}. Set 
$$ B(d,k)=4\br{1-\frac{d(k-1)}{2^{k-1}-1}}.$$
	\begin{enumerate}
		\item If $\rho \in \Bw$ satisfies $\|\rho-\bar\rho\|_2\le n^{-3/8}$ then
		\begin{align}\label{eq_sm_norm_bal}
		\Erw \brk{Z_{\rho}^{(2)}(\hnm)} &\sim C_n(d,k)\exp\brk{nf_2(\bar\rho)-n\frac{B(d,k)}2\|\rho-\bar\rho\|_2^2}.
		\end{align}
		\item There exists $A=A(d,k) >0$ such that if $ \rho \in \Bw$ satisfies $\| \rho - \bar\rho\|_2 >n^{-3/8}$, then
		\begin{align}\label{eq_sm_norm_big}
		\Erw \brk{Z^{(2)}_{\rho}(\hnm)} = O\br{\exp \brk{n f_2\br{\bar\rho}- An^{1/4}}}.
		\end{align}		
	\end{enumerate}
\end{lemma}

\begin{proof}
To prove (\ref{eq_sm_norm_bal}), we observe that if $\rho \in \Bw$  satisfies $\|\rho-\bar\rho\|_2\le n^{-3/8}$, by Taylor expansion around $\bar\rho$ (where $\cH$ and $g_2$ are maximized) we obtain
	\begin{align}
	\label{eq_sm_h}\cH(\rho) &= \cH\br{\bar\rho}-2\|\rho-\bar\rho\|_2^2+ o\br{n^{-1}} \text{ and }\\
	\label{eq_sm_g2}g_2(\rho)&=g_2\br{\bar\rho}-\frac{2d(k-1)}{2^{k-1}-1}\|\rho-\bar\rho\|_2^2+ o\br{n^{-1}}.
	\end{align}
	Inserting this into (\ref{eq_sm_rho_2})  yields (\ref{eq_sm_norm_bal}).

To prove (\ref{eq_sm_norm_big}), we distinguish two cases.

\noindent \textbf{Case 1:} $\|\rho-\bar\rho\|_2=o(1)$: We observe that similarly to (\ref{eq_sm_h}) and (\ref{eq_sm_g2}) there exists a constant $A=A(d,k) >0$ such that
$$ f_2(\rho) \leq f_2\br{\bar\rho} - A\| \rho - \bar\rho \|_2^2 . $$

\noindent Hence, if $\| \rho - \bar\rho \|_2 > n^{-3/8}$ and $\|\rho-\bar\rho\|_2=o(1)$, then
\begin{align}\label{eq_sm_norm_big_1}
 \Erw \brk{Z_{\rho}^{(2)}(\hnm)}= O \br{n^{-3/2}} \exp \brk{n f_2(\rho)} \leq \exp \brk{n f_2\br{\bar\rho}- A n^{1/4}}.
\end{align}

\noindent \textbf{Case 2:} $\|\rho-\bar\rho\|_2=c$ where $c>0$ is a constant independent of $n$:
We consider the function $\bar f_2:\brk{0,\frac 12} \mapsto \mathbb R$  that results from $f_2$ by setting $\rho_{i \star}=\rho_{\star i}=1/2$. This function was introduced by Achlioptas and Moore \cite{Achlioptas} and has been studied at different places in the literature on random hypergraph 2-colouring. The following lemma quantifies the largest possible deviation of $f_2$ and $\bar f_2$.

\begin{lemma}\label{Lem_barf}
Let $\bar f_2: [0,1]\to \mathbb R$ be defined as
$$\bar f_2(\rho)=\ln 2+ \cH\br{2\rho}+\frac dk \ln\br{1-2^{2-k}+2\rho^k+2\br{\frac 12-\rho}^k}.$$
Then for $\rho=(\rho_{ij})\in \Bw$ we have
$$\exp\brk{nf_2(\rho)}\sim\exp\brk{n\bar f_2(\rho_{00})+O\br{\w^2}}.$$
\end{lemma}
\begin{proof}
	For $\rho\in\Bw$ we consider the function 
	$$\zeta(\rho)=f_2(\rho)-\bar f_2(\rho_{00})$$
	and approximate $\zeta(\rho)$ by a Taylor expansion around $\rho=\bar\rho$. As $f_2(\bar \rho)=\bar f_2(\bar \rho_{00})$ and $\frac{\partial{f_2}}{\partial \rho_{ij}}(\bar \rho)=0$ for $i,j \in \{0,1\}$ and $\bar f_2'(\bar\rho_{00})=0$, we have
	$\zeta(\rho)=C\cdot \|\rho-\bar \rho\|_2^2=O\br{\frac{\w}{\sqrt n}}$ for some constant $C$. Thus,
	$$\max_{\rho\in\Bw}|\zeta(\rho)|=O\br{\frac{\w^2}{n}},$$
	yielding the assertion.
	
\end{proof}

In \cite[Lemma 4.11]{Bapst} the function $\bar f_2$ is analysed and it is shown that in the regime $d/k\le 2^{k-1}\ln 2-2$ it takes its {\em global} maximum at $\rho=\bar \rho$ and $\bar f_2(\rho)<\bar f_2(\bar \rho)$ for all $\rho \in\brk{0,\frac 12}$ with $\rho\ne\bar\rho$ independent of $n$. Combining this with \Lem~\ref{Lem_barf} we find that there exists a constant $A'=A'(d,k)>0$ such that
	\begin{align*}
 f_2(\rho) &= f_2\br{\bar \rho} - A' + O\br{\frac{\w^2}{n}},
	\end{align*}
where we used that $f_2(\bar\rho)=\bar f_2(\bar \rho)$. 

\noindent Thus,
\begin{align}\label{eq_sm_norm_1}
\Erw \brk{Z_{\rho}^{(2)}(\hnm)}= O \br{n^{-3/2}} \exp \brk{n f_2(\rho)} \leq \exp \brk{n f_2\br{\bar\rho}- A'n+O\br{\w^2}}.
\end{align}

\noindent As $\exp \brk{n f_2\br{\bar\rho}- A'n+O\br{\w^2}}=o\br{ \exp \brk{n f_2\br{\bar\rho}- A n^{1/4}}}$, equation (\ref{eq_sm_norm_1}) together with (\ref{eq_sm_norm_big_1}) completes the proof of (\ref{eq_sm_norm_big}). 
\end{proof}

\begin{proof}[Proof of \Prop~\ref{Prop_sm_eta}] 
We let $s\in[\w\nu]$. For a $\hat\rho \in \Bwnis$ we have $\| \hat\rho- \bar\rho\|_2 =O\br{\frac{\w}{\sqrt n}}$ and obtain from the first part of \Lem~\ref{Lem_sm_norm} that
	\begin{align}\label{eq_sm_eta_1}
	\Erw \brk{\Zwnis(\hnm)^2} \geq \Erw \brk{Z^{(2)}_{\hat\rho} (\hnm)} \sim C_n(d,k) \exp \brk{n f_2\br{\bar\rho}+O(\w^2)}.
	\end{align}
On the other hand, because $|\Bwni|$ is bounded by a polynomial in $n$, the second part of \Lem~\ref{Lem_sm_norm} yields
\begin{align} \label{eq_sm_eta_2}
\sum_{\rho \in \Bwni: \| \rho - \bar\rho \|_2 > n^{-3/8} } \Erw \brk{Z_{\rho}^{(2)}(\hnm)} &
 = O\br{\exp \brk{n f_2\br{\bar\rho}-A n^{1/4} + O(\ln n)}}. 
\end{align} 
Combining~(\ref{eq_sm_eta_1}) and~(\ref{eq_sm_eta_2}), we obtain
 \begin{align*} \Erw \brk{\Zwni(\hnm)^2} &\sim  \sum_{\substack{\rho \in \Bwnis}} \Erw \brk{Z_{\rho}^{(2)}(\hnm)}= \Erw \brk{\Zwnis(\hnm)} 
\end{align*}
as claimed.

\end{proof}

\subsection{The leading constant}
In this section we compute the contribution of overlap matrices $\rho \in \Bwnis$. In a first step we show that for $\rho\in \Bwnis$ we can approximate $f_2$ by a function $f_2^s$ that results from $f_2$ by (approximately) fixing the marginals $\rho_{i\star}, \rho_{\star j}$ for $i,j\in\{0,1\}$.

\begin{lemma}\label{Lem_sm_shifted}
Let $k \ge 3, \w, \nu \in\mathbb N$ and $C_n(d,k)$ as in (\ref{eq_sm_rho}). For $s\in[\w\nu]$ remember $\rwni$ from (\ref{rwni}). 
Let $f_2^s:\cB \to \mathbb R$ be defined as
$$f_2^s: \rho \mapsto \cH(\rho)+\frac dk\ln\br{1-2{\rwni}^k-2(1-\rwni)^k+\sum_{i,j=0}^1 \rho_{ij}}.$$
Then for $\rho \in \Bwnis$ it holds that
	\begin{align*}
	\Erw \brk{Z_{\rho}^{(2)}(\hnm)} &\sim C_n(d,k)\exp\brk{nf_2^s(\rho)+O\br{\frac{\w}{\nu}}}.
	\end{align*}
\end{lemma}

\begin{proof}
Equation (\ref{eq_sm_rho_2}) of \Lem~\ref{Lem_sm_rho} yields that
\begin{align}\label{eq_sm_shifted}
\Erw \brk{Z_{\rho}^{(2)}(\hnm)} \sim C_n(d,k)\exp\brk{nf_2(\rho)}.
\end{align}
\noindent Analogously to the proof of \Lem~\ref{Lem_barf} we define 
	$$\zeta^s(\rho)=f_2(\rho)- f_2^s(\rho).$$
To bound $\zeta^s(\rho)$ from above for all $\rho \in \Bwnis$, we observe that we can express the function $f_2$ by setting $\rho_{0\star}=\rwni+\alpha$ and $\rho_{\star 0}=\rwni+\beta$, where $|\alpha|, |\beta| \le \frac{1}{\nu\sqrt n}$ and thus 
\begin{align*}
f_2: \rho \mapsto \cH(\rho)+\frac dk\ln\br{1-\br{\rwni+\alpha}^k-\br{\rwni+\beta}^k-\br{1-\rwni-\alpha}^k-\br{1-\rwni-\beta}^k+\sum_{i,j=0}^1 \rho_{ij}}.
\end{align*}
As we are only interested in the difference between $f_2$ and $f_2^s$, we can reparametrise $\zeta^s$ as
\begin{align*}
\zeta^s(\alpha, \beta)=\frac dk \ln\br{\frac{1-\br{\rwni+\alpha}^k-\br{\rwni+\beta}^k-\br{1-\rwni-\alpha}^k-\br{1-\rwni-\beta}^k+\sum_{i,j=0}^1 \rho_{ij}}{1-2{\rwni}^k-2(1-\rwni)^k+\sum_{i,j=0}^1 \rho_{ij}}}.
\end{align*} 
Differentiating and simplifying the expression yields
	$\frac{\partial \zeta^s}{\partial \alpha}(\alpha, \beta), \frac{\partial \zeta^s}{\partial \beta}(\alpha, \beta)=O\br{\frac{\w}{\sqrt n}}$. 
As we are interested in $\rho\in\Bwnis$ and $|\Bwnis|\le \frac{2}{\nu\sqrt n}$ according to the fundamental theorem of calculus it follows for every $s \in [\w\nu]$ that
	$$\max_{\rho\in\Bwnis}|\zeta^s(\rho)|=\int_{-\br{\nu\sqrt n}^{-1}}^{\br{\nu \sqrt n}^{-1}} O\br{\frac{\w}{\sqrt n}} d\alpha=O\br{\frac{\w}{n\nu}}.$$
	
\noindent Combining this with (\ref{eq_sm_shifted}) yields the assertion.
\end{proof}

\begin{proposition}\label{Prop_second_moment_balanced_exact}
	Let $k \ge 3, \w,\nu \in\mathbb N$ and $d^{\prime}(k-1)<\br{2^{k-1}-1}^2$. Then for all $s\in[\w\nu]$ we have
	\begin{align*}
	&\Erw\brk{\Zwnis(\hnm)} \sim_{\nu} \br{|\Awni|\sqrt{\frac{2}{\pi n}}\exp\brk{nf_1\br{\rwni}}}^2\cdot \\
	& \qquad \qquad \qquad \qquad \qquad \qquad \qquad 
	\exp\brk{\frac{d(k-1)}2\frac{2^k-3}{\br{2^{k-1}-1}^2}}\br{1-\frac{d(k-1)}{\br{2^{k-1}-1}^2}}^{-1/2}.
	\end{align*}
\end{proposition}

\begin{proof}
By \Lem~\ref{Lem_sm_shifted} we know that for $\rho \in \Bwnis$ we have
	\begin{align}\label{eq_sm_shifted_1}
	\Erw \brk{Z_{\rho}^{(2)}(\hnm)} &\sim C_n(d,k)\exp\brk{nf_2^s(\rho)+O\br{\frac{\w}{\nu}}}.
	\end{align}
A Taylor expansion of $f_2^s(\rho)$ around 
$$\rho^s=
\begin{pmatrix}
{\rwni}^2 & \rwni(1-\rwni) \\ \br{1-\rwni}\rwni & \br{1-\rwni}^2
\end{pmatrix}$$
while setting $ D(d,k)=4\br{1-\frac{d(k-1)}{\br{2^{k-1}-1}^2}}$ yields
	\begin{align*}
	f_2^s(\rho) &= f_2^s\br{\rho^s}+\Theta\br{\frac{\w}{n}}\left\|\rho-\rho^s\right\|_2-\frac{D(d,k)}2\left\|\rho-\rho^s\right\|_2^2+ o\br{n^{-1}}.
	\end{align*}
Combining this with (\ref{eq_sm_shifted_1}) we find that
\begin{align}\label{eq_sm_shifted_2}
	\Erw \brk{Z_{\rho}^{(2)}(\hnm)} &\sim C_n(d,k)\exp\brk{nf_2^s\br{\rho^s}+\Theta\br{\w}\left\|\rho-\rho^s\right\|_2-n\frac{D(d,k)}2\left\|\rho-\rho^s\right\|_2^2+O\br{\frac{\w}{\nu}}}.
\end{align}

\noindent For  $\rho^0, \rho^1 \in \Bwni$, we introduce the set of overlap matrices 
$$\Bwnisr = \{ \rho \in \Bwnis: \rho_{\nix\star}= \rho^0, \rho_{\star\nix} = \rho^1 \}.$$
In particular, $\Bwnisr$ contains the ``product'' overlap $\rho^0 \otimes \rho^1$ defined by
$ \br{\rho^0 \otimes \rho^1}_{ij} = \rho^0_i \rho^1_j . $
With these definitions we see that

\begin{align}\label{eq_sm_comp} 
 \Erw \brk{\Zwnis(\hnm)}= \sum_{ \rho^0,\rho^1 \in \Bwni} \sum_{\rho \in \Bwnisr} \Erw \brk{Z_{\rho}^{(2)}(\hnm)}. 
\end{align}

\noindent Let us fix from now on two colour densities $\rho^0, \rho^1 \in \Bwni$.  We simplify the notation by setting 
$$ \widehat{\cB} = \Bwnisr, \qquad \widehat{\rho} = \rho^0 \otimes \rho^1. $$
Thus, we are going to evaluate $$\cS_1=\sum_{\rho \in \widehat{\cB}} \Erw \brk{Z_{\rho}^{(2)}(\hnm)}.$$

\noindent We define the set $\cE_n = \cbc{\boldsymbol\eps=(\eps,-\eps,-\eps,\eps), \eps \in \frac1n\mathbb{Z}, 0\le\eps\le 1}$. Then for each $\rho\in\widehat{\cB}$ we can find $\boldsymbol\eps\in \cE_n$ such that 
\begin{align*}
 \rho=\widehat\rho+\boldsymbol \eps
\end{align*}
Hence, this gives $\left\|\rho-\rho^s\right\|_2=\left\|\widehat\rho+\boldsymbol \eps-\rho^s\right\|_2$ and the triangle inequality yields
\begin{align*}
\left\|\boldsymbol \eps\right\|_2-\left\|\widehat\rho-\rho^s\right\|_2\le \left\|\widehat\rho+\boldsymbol \eps-\rho^s\right\|_2 \le \left\|\boldsymbol \eps\right\|_2+\left\|\widehat\rho-\rho^s\right\|_2.
\end{align*}
As $\left\|\widehat\rho-\rho^s\right\|_2\le\frac{1}{\nu\sqrt n}$ and for $\nu\to\infty$ it holds that $\frac{1}{\nu\sqrt n}=o(n^{-1/2})$, in this case we have
\begin{align}\label{eq_sm_norm}
 \left\|\rho-\rho^s\right\|_2=\left\|\boldsymbol \eps\right\|_2+o(n^{-1/2}).
\end{align}
Observing that $f_2^s\br{\rho^s}=\br{f_1(\rwni)}^2$ and inserting (\ref{eq_sm_norm}) into (\ref{eq_sm_shifted_2}), we find
\begin{align} \label{eq_sm_S1}  
\cS_1
\nonumber & \sim_{\nu}  \sum_{\substack{\rho \in \widehat{\cB}}} C_n(d,k)  \exp \brk{n f_2^s\br{\rho^s}- n \frac{D(d,k)}{2} \left\|\boldsymbol \eps\right\|_2^2+ o(n^{1/2})\left\|\boldsymbol \eps\right\|_2+o(1)}\\
	&  \sim_{\nu} C_n(d,k) \exp \brk{2 n f_1^s\br{\rwni}}\sum_{\substack{\rho \in \widehat{\cB}}}  \exp \brk{- n \frac{D(d,k)}{2}\left\|\boldsymbol \eps\right\|_2^2+o(n^{1/2})\left\|\boldsymbol \eps\right\|_2 }. 
\end{align}
It follows from the definition of $\widehat{\cB}$ that 
 $$\left \{ \widehat{\rho} + \boldsymbol \eps:  \boldsymbol \eps\in \cE_n, \| \boldsymbol \eps \|_2  \leq {n^{-3/8}}/2  \right \} \subset 
 \left \{ \rho \in \widehat{\cB}\right \} 
 \subset \left \{ \widehat{\rho} + \boldsymbol \eps:   \boldsymbol \eps \in \cE_n \right  \}. $$
As
 \begin{align*}
 \cS_2&\sim_{\nu} C_n(d,k) \exp \brk{n f_2^s\br{\rho^s}}
 \sum_{\boldsymbol\eps \in \cE_n ,\;\| \boldsymbol\eps\|_2 > n^{-3/8}/2} \exp \brk{- n \frac{D(d,k)}{2}\left\|\boldsymbol \eps\right\|_2^2 (1+o(1))} \\
 &\le C_n(d,k) \exp \brk{n f_2^s\br{\rho^s}}O(n)\exp \brk{- \frac{D(d,k)}{8} n^{1/4}}, 
 \end{align*}
equation (\ref{eq_sm_S1}) yields $\lim_{\nu \to \infty} \lim_{n \to \infty} \cS_2/\cS_1=0$ and we see that $\boldsymbol \eps \in \cE_n$ with $\| \boldsymbol\eps\|_2 > n^{-3/8}/2$ do only contribute negligibly. Thus, 
we conclude, using the formula of Euler-Maclaurin and a Gaussian integration, that
	\begin{align}\label{eq_sm_S1_2}
	\nonumber \cS_1&\sim_{\nu} C_n(d,k) \exp \brk{2 n f_1^s\br{\rwni}}\sum_{\boldsymbol \eps \in \cE_n}  \exp \brk{- n \frac{D(d,k)}{2}\left\|\boldsymbol \eps\right\|_2^2+o(n^{1/2})\left\|\boldsymbol \eps\right\|_2 }\\
	\nonumber 	&\sim_{\nu}  C_n(d,k) \exp \brk{2 n f_1^s\br{\rwni}} n \int \exp\brk{-n\frac{D(d,k)}{8}{\eps}^2+o(n^{1/2})\eps}d\eps\\ 
		& \sim_{\nu} C_n(d,k)\exp \brk{2 n f_1^s\br{\rwni}} \sqrt{\frac{\pi n}8} \br{1-\frac{d(k-1)}{\br{2^{k-1}-1}^2}}^{-1/2}.
	\end{align}
In particular, the last expression is independent of the choice of the vectors $\rho^0, \rho^1$ that defined $\widehat{\cB}$.
Therefore, substituting~(\ref{eq_sm_S1_2}) in the decomposition (\ref{eq_sm_comp}) completes the proof.
 \end{proof}

 \begin{proof}[Proof of \Prop~\ref{Prop_ratio_second_first}]
From \eqref{eq_lambdadeltainfty} we remember that
\begin{align}\label{eq_together_1}
 \exp \brk{\sum_{l \geq 2} \lambda_l \delta_l^2} = \exp\brk{-\frac{d(k-1)}2\frac 1{\br{2^{k-1}-1}^2}}\br{1-\frac{d(k-1)}{\br{2^{k-1}-1}^2}}^{-1/2}.
\end{align}
To prove \Prop~\ref{Prop_ratio_second_first} we combine \Lem~\ref{Lem_fm_shifted} with \Prop s~\ref{Prop_sm_eta} and \ref{Prop_second_moment_balanced_exact} yielding
\begin{align}\label{eq_together_2}
\frac{ \Erw \brk{\Zwni(\hnm)^2}}{\Erw \brk{\Zwni(\hnm)}^2}
\nonumber &\sim_{\nu} \exp\brk{\frac{d(k-1)}2\br{\frac{2^k-3}{\br{2^{k-1}-1}^2}-\frac{2}{2^{k-1}-1}}}\br{1-\frac{d(k-1)}{\br{2^{k-1}-1}^2}}^{-1/2}\\
&= \exp\brk{-\frac{d(k-1)}2\frac 1{\br{2^{k-1}-1}^2}}\br{1-\frac{d(k-1)}{\br{2^{k-1}-1}^2}}^{-1/2}.
\end{align}
Combining equations (\ref{eq_together_1}) and (\ref{eq_together_2}) completes the proof.
\end{proof}

\subsubsection*{Acknowledgements}
I thank my PhD supervisor Amin Coja-Oghlan for constant support and valuable suggestions and Victor Bapst and Samuel Hetterich for helpful discussions.

\end{document}